\documentclass{amsart}

\usepackage[left=20mm,right=20mm,top=25mm,bottom=25mm]{geometry} 
\usepackage{amsfonts}

\usepackage{url}
\usepackage{amssymb}
\usepackage{amsmath}
\usepackage{stmaryrd}
\usepackage{siunitx}
\usepackage{commath}
\usepackage{graphicx}
\usepackage[caption=false]{subfig}
\usepackage{enumerate}
\usepackage{enumitem}
\usepackage{todonotes}
\usepackage{pgfplotstable, pgfplots}
\usepackage{cleveref}
\usepackage{bm} 
\usepackage{rotating} 

\newcommand{\divergence}{\operatorname{div}}

\newcommand{\curl}{\operatorname{curl}}

\newcommand{\rr}{\mathbb{R}}
\newcommand{\mm}{{\mathbb{M}}}
\newcommand{\dd}{{\mathbb{D}}}
\newcommand{\kk}{{\mathbb{K}}}

\newcommand{\vecb}[1]{{#1}}
\newcommand{\matb}[1]{{#1}}

\newcommand{\eps}{{\varepsilon}}

\newcommand{\HT}{\widetilde{H}}
\newcommand{\AT}{\widetilde{A}}

\newcommand{\BDM}{{\textrm{BDM}}}

\newcommand{\trans}{{{T}}}

\newcommand{\normal}{{\vecb{n}}}
\newcommand{\tangential}{{\vecb{t}}}

\newcommand{\ds}{}

\newcommand{\dx}{}

\DeclareMathOperator*{\argmin}{arg\,min}

\newcommand{\mesh}{\mathcal{T}_h}

\newcommand{\VUUH}{\mathcal{V}_h}
\newcommand{\LOVUUH}{\mathcal{V}_h^{\text{lo}}}
\newcommand{\VUUHC}{\mathcal{V}_h^c}
\newcommand{\VUUHI}{\mathcal{V}_h^{\circ}}
\newcommand{\VUUHH}{\mathcal{V}_h^{\text{harm}}}

\newcommand{\trace}[1]{\textrm{tr}({#1})}
\newcommand{\jump}[1]{ {[\![ #1 ] \!]}}
\newcommand{\mean}[1]{{ \{\!\{#1\}\!\} }}
\newcommand{\Dev}[1]{\operatorname{dev}{(#1)}}
\newcommand{\dev}[1]{\Dev{#1}}

\newcommand{\Velvar}{{\vecb{u}}}

\newcommand{\Stressvarhtest}{\matb{\tau}_h}

\newcommand{\Poly}{{\mathbb{P}}}

\newcommand{\projlo}{\Pi^{\operatorname{lo}}}
\newcommand{\spacedim}{{d}}

\renewcommand{\div}{\operatorname{div}}

\def\addlegendimage{\csname pgfplots@addlegendimage\endcsname}

\newcommand{\Vhat}{\hat{V}}

\newcommand{\Vbar}{\bar{V}}
\newcommand{\ub}{\bar{u}}
\newcommand{\ubh}{\ub_h}
\newcommand{\vb}{\bar{v}}
\newcommand{\vbh}{\vb_h}

\newcommand{\vof}[1]{\bm{#1}}
\newcommand{\mat}[1]{\bm{#1}}

\newcommand{\Ss}{A}                            
\newcommand{\SsD}{A^{\partial}}                
\newcommand{\SsM}{\mat{A}}                     
\newcommand{\SsDM}{\mat{A}^{\partial}}         
\newcommand{\SsMH}{\hat{\mat{A}}}                   
\newcommand{\SsDMH}{\hat{\mat{A}}^{\partial}}       
\newcommand{\SsMHA}{\hat{\mat{A}}_{a}}              
\newcommand{\SsDMHA}{\hat{\mat{A}}_{a}^{\partial}}  
\newcommand{\SsMHM}{\hat{\mat{A}}_{m}}               
\newcommand{\SsDMHM}{\hat{\mat{A}}_{m}^{\partial}}   

\newcommand{\Abm}{\bar{\mat{A}}}

\newcommand{\GalEq}{\sim_{\scriptscriptstyle \text{G}}} 
\newcommand{\GalEqD}{\sim_{\scriptscriptstyle \text{G}}^{\partial}} 

\newcommand{\mM}{\mat{M}}
\newcommand{\mDM}{\mat{M}^{\partial}}

\newcommand{\mA}{\mat{A}}

\newcommand{\mE}{\mat{E}}
\newcommand{\mDE}{\mat{E}^{\partial}}
\newcommand{\mDET}{\mat{E}^{\partial,T}}
\newcommand{\mC}{\mat{C}}

\newcommand{\mB}{\mat{B}}
\newcommand{\mK}{\mat{K}}
\newcommand{\mKH}{\hat{\mat{K}}}
\newcommand{\mZ}{\mat{0}}
\newcommand{\mI}{\mat{I}}

\newcommand{\mSp}{\mat{S}_{p}}
\newcommand{\mSpH}{\hat{\mat{S}}_{p}}

\newcommand{\vsig}{\vof{\sigma}}
\newcommand{\vu}{\vof{u}}
\newcommand{\vuh}{\vof{\hat{u}}}
\newcommand{\vw}{\vof{w}}
\newcommand{\vomega}{\vof{\omega}}
\newcommand{\vx}{\vof{x}}
\newcommand{\vp}{\vof{p}}
\newcommand{\vq}{\vof{q}}
\newcommand{\vf}{\vof{F}}

\newcommand{\pdg}{{k}}   

\newcommand{\IU}{\mathcal{I}_{\Vbar_h}}
\newcommand{\IV}{\mathcal{I}}
\newcommand{\IVf}{\mathcal{I}_{\scriptscriptstyle f}}

\newcommand{\GammaD}{\Gamma_D}
\newcommand{\GammaN}{\Gamma_N}
\newcommand{\GammaNT}{\Gamma_{\tilde{N}}}
\newcommand*{\verts}{\mathcal{V}}
\newcommand*{\meshT}{\mathcal{T}_{h, T}}
\newcommand{\RM}{{\textbf{RM}}}

\newcommand{\facetsT}{\mathcal{F}_T}

\newcommand{\facets}{\mathcal{F}_h}
\newcommand{\facetsD}{\facets^{\scriptscriptstyle D}}
\newcommand{\facetsN}{\facets^{\scriptscriptstyle N}}
\newcommand{\facetsNT}{\facets^{\scriptscriptstyle \tilde{N}}}
\newcommand{\facetsI}{\facets^{0}}

\usepackage{booktabs}
\usepackage{siunitx}
\newtheorem{theorem}{Theorem}
\newtheorem{lemma}{Lemma}

\newtheorem{corollary}{Corollary}
\newtheorem{remark}{Remark}

\author[L. Kogler]{Lukas Kogler}
\address{Institute for Analysis and ScientificComputing, TU Wien,
Wiedner Hauptstrasse 8-10, 1040 Vienna, Austria}
\email{lukas.kogler@tuwien.ac.at}

\author[P.~L.~Lederer]{Philip L. Lederer}
\address{Institute for Analysis and ScientificComputing, TU Wien,
Wiedner Hauptstrasse 8-10, 1040 Vienna, Austria}
\email{philip.lederer@tuwien.ac.at}

\author[J.Sch\"oberl]{joachim Sch\"oberl}
\address{Institute for Analysis and ScientificComputing, TU Wien,
Wiedner Hauptstrasse 8-10, 1040 Vienna, Austria}
\email{joachim.schoeberl@tuwien.ac.at}

\title[A conforming auxiliary space preconditioner for the MCS method]{A conforming auxiliary space preconditioner for the mass conserving mixed stress method}

\begin{document}
\maketitle
\begin{abstract}
  We are studying the efficient solution of the system of linear
  equation stemming from the mass conserving mixed stress (MCS)
  method discretization of the Stokes equations. To that end we
  perform static condensation to arrive at a system for the pressure
  and velocity unknowns. An auxiliary space preconditioner for the
  positive definite velocity block makes use of efficient and scalable
  solvers for conforming Finite Element spaces of low order and is
  analyzed with emphasis placed on the polynomial degree of the
  discretization. Numerical experiments demonstrate the potential of
  this approach and the efficiency of the implementation.
\end{abstract}

\section{Introduction}
\label{sec::introduction}
Let $\Omega \subset \rr^d$ be a bounded domain with $d=2$ or $3$
with Lipschitz boundary $\Gamma := \partial \Omega$.
Let $u$ and $p$ be the velocity and the pressure,
respectively.
Given an external body force $f:\Omega \to \rr^d$ and
the double of kinematic viscosity denoted by ${\nu}$, the
velocity-pressure formulation of the Stokes system is given by
\begin{subequations}\label{eq::stokes}
\begin{align}
  -\divergence({\nu} \eps(u)) + \nabla p & = f \quad \textrm{in } \Omega,  \label{eq::stokes_A}\\
  \divergence (u) &=0 \quad  \textrm{in } \Omega, \label{eq::stokes_div}
\end{align}
\end{subequations}
where $\eps(\Velvar) = \frac{1}{2}(\nabla u + (\nabla u)^ \trans)$. By
introducing additional matrix valued variables $\sigma := -\nu\eps(u)$
for the stress and $\omega := \frac{1}{2}(\nabla u - (\nabla u)^T)$,
these equations can be restated as
\begin{subequations}
  \label{eq::mixedstressstokes}
  \begin{alignat}{2}
    -\nu^{-1}\dev{\sigma} - \nabla u + \omega & = 0 \quad \textrm{in } \Omega, \label{eq::mixedstressstokes-a} \\
    \divergence(  \sigma ) + \nabla p & = f \quad \textrm{in } \Omega, \label{eq::mixedstressstokes-b} \\
    \sigma - \sigma^\trans & = 0 \quad \textrm{in } \Omega, \label{eq::mixedstressstokes-c} \\
  \divergence (u) &=0 \quad \textrm{in } \Omega, \label{eq::mixedstressstokes-d} 
\intertext{
where \eqref{eq::mixedstressstokes-a} is motivated by the fact that
for the solution of \eqref{eq::stokes} we have $ \sigma = -\nu \eps(u)
= -\nu \dev{\eps(u)} = \dev{\sigma}$.
The introduction of $\omega$ as a Lagrange multiplier enables 
the derivation of discrete methods that enforce the symmetry constraint
\eqref{eq::mixedstressstokes-c} weakly, see also \cite{Stenberg1988, MR2629995, MR2336264}.
As boundary conditions, we consider Dirichlet ones for the velocity $u$,
homogenous purely for clarity of the presentation,
and two kinds of outlet conditions,
}
    u &= 0 \quad \textrm{on } \GammaD, \label{eq::stokes_D}\\
    (\sigma + p I) n &= 0 \quad \textrm{on } \GammaN, \label{eq::stokes_N}\\
    ((\sigma + p I) n) \cdot n = u_t &= 0 \quad \textrm{on }\GammaNT, \label{eq::stokes_NT}
  \end{alignat}
\end{subequations}
where $I$ is the $d\times d$ identity matrix and $u_t$ is the tangential part of $u$.
We assume that both $\GammaD$ and at least one of $\GammaN$ or $\GammaNT$
have positive measure.
As usual, when $\GammaN=\GammaNT=\emptyset$, an additional condition must be imposed
on the pressure to make it unique.

In recent years, divergence-free and pressure-robust Finite Element discretizations,
that is those whose solutions
fulfill \eqref{eq::mixedstressstokes-d} strongly, and
allow for pressure-independent a-priori error estimates respectively,
have been of great interest \cite{JLMNR:sirev}.

For the velocity-pressure formulation \eqref{eq::stokes}, one class of such
methods are certain Hybrid Discontinuous Galerkin (HDG) methods that
take the velocity in $H(\div,\Omega)$ and the pressure in $L^2(\Omega)$,
i.e. they only build normal continuity into the Finite
Element space while the tangential continuity of the solution is enforced
via Lagrange parameters.
To make the resulting system for the velocity positive definite,
a consistent stabilization term has to be added, often involving
either a parameter that has to be sufficiently large or a lifting
of the jump, see \cite{MR2684358,arnold2002unified}.

In \cite{lederer2019mass, Lederer:2019c}, the authors presented a
novel variational formulation for the Stokes equations that still
takes the velocity in $H(\div, \Omega)$ and pressure in $L^2(\Omega)$,
remaining the property of exactly divergence-free and pressure-robust
solutions, but is based on \eqref{eq::mixedstressstokes} instead of
\eqref{eq::stokes}. This mass conserving mixed stress (MCS) method
features a normal-tangential continuous stress space and requires no
stabilizing term. It was already remarked in the original work
\cite{lederer2019mass} that static condensation can be performed to
eliminate certain $\sigma$ degrees of freedom (dofs) and later in
\cite{gopalakrishnan2021minimal} this approach was taken to it's
logical conclusion of breaking the normal-tangential continuity of
$\sigma$ with a Lagrange parameter and eliminating $\sigma$ entirely.
The resulting, condense, system is one for the velocity in $H(\div,
\Omega)$, the pressure in $L^2(\Omega)$, and the newly introduced
Lagrange parameter $\hat{u}$; It turns out to be an approximation to
the tangential velocity trace on the mesh facets. The velocity
unknowns $u,\hat{u}$ take the place of $\sigma$ as primal variables in
the condense saddle point system, with the pressure remaining the
Lagrange parameter enforcing \eqref{eq::mixedstressstokes-d}. That is,
the condense system involves the same variables, and has the same
structure as the HDG methods mentioned above, \emph{but without the
need for a stabilization term}. As the first contribution of this work
we take a closer look at the condense system and in particular proof
that the velocity block is in fact positive definite, as was claimed
in \cite{gopalakrishnan2021minimal} for a low order MCS method, and is
related to the velocity block stemming from an HDG method with optimal
stabilization.

We then move on to the question of how to efficiently solve the condense system
and consider preconditioned Krylov space methods.
Preconditioning techniques for saddle point systems
based on separate preconditioners for the primal (velocity) and
Lagrange (pressure) unknowns are a well studied subject, see
\cite{benzi_golub_liesen_2005}, and the pressure Schur complement
is easily preconditioned, see \cite{Wathen1993FastIS}.
Therefore, our focus is on identifying and analyzing suitable
preconditioners for the condense velocity block.

The literature on preconditioners for conforming methods is vast
and includes, among others, domain decomposition, see \cite{Toselli2005}, as well
as Geometric, see \cite{braess}, and Algebraic, see \cite{xu_zikatanov_2017}, Multigrid methods
and an even somewhat comprehensive review
would be beyond the scope of this work.
We will take as given that efficient and scalable solvers for conforming methods
exist and are available.

Preconditioners for HDG methods are not quite as well studied in literature,
one recurring theme is the attempt to reuse conforming preconditioners for these
non-conforming spaces.
For example, a non-nested Multigrid method with conforming coarse grid spaces
was studied in \cite{mg_hdg},
and auxiliary space preconditioners
(ASP, see \cite{XuAux_1996}) that also feature a conforming sub-space
were considered in \cite{Fu2021}.

The idea at the heart of both approaches
is to decompose functions in the non-conforming space into
a conforming component plus a (small) remainder and to address them
separately with some pre-existing conforming preconditioner
and a simple, computationally inexpensive method such as
(Block-)Jacobi, respectively.

The principal focus in this work is on the introduction and analysis of ASPs for the MCS method.
The main improvement over the theory in \cite{Fu2021} is that the analysis of the
velocity preconditioners extends techniques from \cite{Schoeberl2013} and is
explicit in the polynomial degree of the discretization.
In particular, the main result, Theorem \ref{theorem::aux_pc_schur},
states that the condition number of a particular ASP is bounded
by $\gamma\cdot(\log(k))^3$, where $\pdg$ is the polynomial degree
of the discretization and $\gamma$ is a constant stemming from
the relation between condense MCS and HDG norms.

We close out the discussion with numerical experiments that
demonstrate the robustness and scalability of the proposed preconditioners.
It is a testament to the elegance and simplicity of the ASP method
that we were able to scale the computations to a relatively large scale
by leveraging existing, scalable and highly performant software.

\paragraph{Outline}

We gather notation used throughout this work in Section \ref{sec::notation}
and introduce various Finite Element spaces and norms in Section \ref{sec::fem_nreq}
which also contains some useful technical results.
Section \ref{sec::mcsstokes} reviews the MCS method itself and contains a
thorough discussion of static condensation as well as results on the
obtained condense systems.
Approaches for preconditioning saddle point matrices with separate
preconditioners for the primal unknowns and Lagrange multipliers
as well as the method of auxiliary space preconditioning
are recalled in Section \ref{sec::pc_frame}.
The main results can be found in Section \ref{sec::aux},
where different variations of ASPs for the velocity block of the
Stokes system are discussed.
In Section \ref{sec::lorder}, we sketch the treatment of the lowest order case
which is not covered by the theory developed in previous sections.
Finally, numerical experiments are performed in Section \ref{sec::numerical}.

\section{Notation} \label{sec::notation}
With $\mm$ denoting the vector space of real $d \times d$ matrices, we define
the subsets of skew-symmetric and skew-symmetric trace-free matrices by
\begin{align*}
  \kk =\{\tau \in \mm: \tau + \tau^\trans = 0 \} \quad \textrm{and} \quad   \dd =\{\tau \in \mm: \tau:I = 0 \},
\end{align*}
where $(\cdot)^\trans$ denotes the transpose and $I \in \mm$ the
identity matrix. To differentiate between scalar-, vector- and
matrix-valued functions on some subset $D\subseteq\Omega$ we include
the range in the notation for the latter two while we omit it for the
former one, i.e. where $L^2(D, \rr) = L^2(D)$ denotes the space of
square integrable $\rr$-valued scalar functions, the spaces $L^2(D,
\rr^d)$ and $L^2(D, \mm)$ denote the analogous vector- and
matrix-valued spaces. Similarly, $\Poly^{\pdg}(D,\rr)=\Poly^k(D)$,
etc., denote the set of scalar-, vector- or matrix- valued polynomials
up to degree $\pdg$ on $D$. We use the notation $(\cdot,\cdot)_{D}$
for the $L^2$-inner product on $D$ and set $\| \cdot \|^2_D = (\cdot
,\cdot)_D$. The $L^2$-orthogonal projection onto
$\Poly^{\pdg}(D,\cdot)$ (the range should be clear from context) is
denoted by $\Pi^{\pdg}_{D}$ and we will occasionally omit the
subscript. Similarly, the $L^2$-orthogonal projector onto the
(restrictions to $D$ of) the rigid body modes $R_D := \{ u(x) = a +
b\times x : a,b\in\rr^d \}$ is written as $\Pi^R_{D}$.

In the following, let $\phi$, $\Phi$, and $\Psi$ be smooth scalar-,
vector-, and matrix-valued functions, respectively.
The operator $\nabla$ is to be understood from context as resulting in either
in a vector whose components are $\partial_i \phi := \partial\phi / \partial x_i$ or a matrix with
components $(\partial_i \Phi_j)$.
For vector-valued functions in three dimensions the operator $\curl$ is defined as
$\curl \Phi := \nabla \times \Phi$
and in two we understand it to refer to the scalar-valued
$\curl \phi := -\partial_2 \phi_1 + \partial_1 \phi_2$.
The divergence operator $\div$ is understood as
$\div\Phi := \sum_{j=1}^d \partial_j \Phi_{j}$ for vectors
and is applied row-wise to matrices, i.e.  $(\div\Psi)_i := \sum_{j=1}^d \partial_j \Psi_{ij}$.
Besides the well known trace operator $\trace{\Psi} := \sum_{j=1}^d \Psi_{ii}$ and
the deviatoric part $\dev{\Psi} := \Psi - \frac{1}{d} \trace{\Psi}I$ we further
introduce the operator $\kappa: \rr^{d(d-1)/2} \to \kk$ by
\begin{align*}
  \kappa(\phi) :=
  \frac 1 2 
  \begin{pmatrix}
    0 & -\phi \\
    \phi & 0         
  \end{pmatrix}
  \; \text{ if } d=2,
  \qquad
  \kappa(\Phi) :=
  \frac 1 2 
  \begin{pmatrix}
    0 & -\Phi_3 & \Phi_2 \\
    \Phi_3 & 0 & -\Phi_1 \\
    -\Phi_2 & \Phi_1 & 0
  \end{pmatrix}
  \;\text{ if } d =3.  
\end{align*}

Based on these differential operators,
we use standard notation for the Sobolev spaces $H^m(\Omega,\rr)=H^m(\Omega), H(\div, \Omega)$
and $H(\curl, \Omega)$ with $m \ge 0$.
Further, for some $\Gamma_*\subseteq\partial\Omega$, a subscript "$0, \Gamma_*$"
indicates that the corresponding natural traces vanish on $\Gamma_*$, and
we use only the zero subscript if $\Gamma_* = \partial \Omega$.

We denote by $\mesh$ a quasi-uniform and shape regular
triangulation of the domain $\Omega$ into tetrahedra. 
Let $h$ denote the maximum of the diameters of all elements in $\mesh$.
The set of element interfaces and boundaries, or facets, is denoted by $\facets$
and the set of facets of a particular element $T\in\mesh$ is $\facetsT := \{F\in\facets : F\subseteq\partial T\}$.
By an abuse of notation, we shall also use $\facets$ to denote the domain formed by union of all $F \in \facets$.
We assume that the mesh resolves the domain boundary parts in the sense that
$\forall F\in\facets $ with $F\subseteq\partial\Omega~\exists!\Gamma_*\in\{\GammaD,\GammaN,\GammaNT\} $ such that $F\subseteq\Gamma_*$.
This splits $\facets$ into boundary facets
$\facetsD := \{F\in\facets : F\subseteq\GammaD\}$,
$\facetsN := \{F\in\facets : F\subseteq\GammaN\}$, and 
$\facetsNT := \{F\in\facets : F\subseteq\GammaNT\}$,
and interior facets
$\facetsI := \facets\setminus (\facetsD\cup\facetsN\cup\facetsNT)$.
According to this mesh we also introduce the ``broken'' spaces
\begin{align*}
  H^m(\mesh,\cdot) := \prod_{T \in \mesh} H^m(T,\cdot),
  \quad
  \Poly^{\pdg}(\mesh,\cdot) :=  \prod_{T \in \mesh} \Poly^{\pdg}(T,\cdot),
  \quad
  \Poly^{\pdg}(\facets,\cdot) := \prod_{F\in\facets}\Poly^{\pdg}(F,\cdot),
\end{align*}
where, as before, we include the range explicitly
e.g. as in $\Poly^{\pdg}(\mesh, \rr^d)$.
On each $F \in \facets$ we denote by $\jump{\cdot}$ and $\mean{\cdot}$
the standard jump and mean value operators and take them to be the
identity on boundary facets.
On each element boundary and each facet $F\in\facets$
we denote by $n$ the outward unit normal vector.
The scalar normal and vector-valued tangential traces of a
sufficiently smooth function $v$ are given by 
$v_n := v\cdot n$ and $v_t := v - v_n n$.
Similarly, the normal-normal and normal-tangential traces of a
smooth matrix-valued function $\Psi$ are $\Psi_{nn}: = \Psi : (n \otimes n) = n^{\trans} \Psi n $
and $\Psi_{nt} = \Psi n - \Psi_{nn} n$.

We write functions in general Sobolev spaces as $u,\hat{u},\omega$,
etc., discrete functions with a subscript $h$ as $u_h, \hat{u}_h, \omega_h$, etc.,
and their via Galerkin isomorphism identified coefficient vectors w.r.t to some given Finite Element basis
as $\vof{u}, \hat{\vof{u}}, \vof{\omega}$, etc.
For readability of the presentation we make no difference between row and column vectors
and, for example, write $(\vu,\vuh)$ for the coefficient vector of $(u_h,\hat{u}_h)$
which should strictly speaking be the column vector $(\vu^T, \vuh^T)^T$.
Similarly, operators are capital letters $A,B$, etc.,
their discrete counterparts $A_h, B_h$, etc., and the
corresponding Finite Element matrices $\mat{A}, \mat{B}$, etc.
Occasionally, when it is useful to emphasize the Galerkin isomorphism
we use $\GalEq$, e.g. $u_h\GalEq\vof{u}$ or $A_h\GalEq\mat{A}$.

Finally, throughout this work we write $A \lesssim B$ when there
exists a constant $c>0$ {\em independent of the mesh size $h$ and the viscosity $\nu$}
such that $c A \le B$ and $A\sim B \Leftrightarrow A\lesssim B \wedge B \lesssim A$.
For example, due to quasi-uniformity we have $h \sim \textrm{diam}(T)~\forall T\in \mesh$.
For two elliptic operators $A,B$ (or symmetric and positive definite matrices $\mat{A},\mat{B}$)
we take $A\lesssim B$ to mean that the maximum eigenvalue of the generalized eigenvalue problem
$A x = \lambda B x$ is bounded by a constant $C$ similarly independent of $h$ and $\nu$.
Note that in inequalities related to discrete functions or operators,
unless explicitly stated otherwise, these constants can depend on
the polynomial degree.
Henceforth we assume that $\nu$ is a constant.

\section{Finite Elements and norm equivalences} \label{sec::fem_nreq}
Reminding ourselfs that the lowest order case is addressed separately in Section \ref{sec::lorder},
we define the following approximation spaces for $\pdg\ge 2$:
\begin{align}
  V_h &:= \{u_h \in \BDM^{\pdg}(\mesh): (u_h)_n = 0 \textrm{ on } \GammaD \}, \label{eq::fes_V} \\
  \hat{V}_h &:= \{ \hat u_h \in \Poly^{\pdg-1}(\facets, \rr^d): (\hat{u}_h)_n = 0 ~ \forall F \in \facets \textrm{ and } \hat u_h = 0 ~ \forall F \subset \GammaD\cup\GammaNT \}, \label{eq::fes_Vh} \\
  W_h & := \Poly^{\pdg-1}(\mesh, \kk), \label{eq::fes_W} \\
  \Sigma_h &:= \{ \tau_h \in \Poly^{\pdg}(\mesh, {\dd}):(\tau_h)_{\normal\tangential} \in \Poly^{\pdg-1}(F, \rr^d) ~\forall F \in \facets \}, \label{eq::fes_Sigma} \\
  Q_h & := \Poly^{\pdg-1}(\mesh, \rr), \label{eq::fes_Q} \\
  \Vbar_h & := \{ u_h\in \Poly^1(\mesh, \rr^d)\cap H^1(\Omega, \rr^d) : u_h = 0 \textrm{ on } \GammaD \} \label{eq::fes_Vbar}.
\end{align}
See \cite{brezzi2012mixed} for a detailed discussion of the $H(\div)$-conforming Brezzi-Douglas-Marini ($\BDM$)
space appearing in the definition of $V_h$.
Note that, restricted to a single element $T$,
in addition to $\Poly^{\pdg-1}(T, \dd)$, the
stress space $\Sigma_h$ also includes functions in $\mathbb{P}^{\pdg}(T, \dd)$
with vanishing normal tangential trace (``$nt$-bubbles'').
We further define the space of divergence
free velocities $ V_h^0 := \{v_h \in V_h: \div(v_h) = 0 \}$ and the
product spaces $ \VUUH := V_h \times \hat V_h$,
$U_h := V_h \times \hat V_h \times W_h$ and
$U^0_h := V^0_h \times \hat V_h \times W_h$.
Following \cite{Schb_DDHDG}, for $T\in\mesh$, $F\in\facetsT$ and $u\in\Poly^{\pdg}(F,\rr^d)$ we write
\begin{align}\label{eq::jFjumpN}
  \|u\|_{j,F,T}^2 := \sup_{\sigma\in\Poly^{\pdg}(T,\rr^d)}\frac{(u, \sigma)_F^2}{\|\sigma\|_T^2} \sim h^{-1}\sum_{j=0}^{\pdg}\pdg(\pdg-j+1)\|(\Pi^j_F - \Pi_F^{j-1})u\|_F^2,
\end{align}
where $\Pi^{-1}_F:=0$ and the equivalence was shown in \cite[Theorem 2]{Schb_DDHDG}.
Note that where it is clear from context which volume element $T$ is meant,
we omit it from the subscript and simply write $\|\cdot\|_{j,F}$.
We define Hybrid Discontinuous Galerkin (HDG) norms on $\VUUH$ and $U_h$ by
\begin{align}
  \|(u_h, \hat{u}_h)\|_{\eps,h}^2 &:= \sum_{T\in\mesh}\Big(\|\eps(u_h)\|_T^2+ \sum_{F\in\facetsT}\|\Pi^{\pdg-1}(u_h - \hat{u}_h)_t\|_{j,F}^2\Big), \label{eq::nhdg_eps} \\
  \| (u_h, \hat u_h, \omega_h) \|^2_{U_h} 
    &:= \sum_{T\in\mesh}\|\eps(u_h)\|_T^2 +\| \kappa(\curl(u_h)) - \omega_h \|_T^2 + h^{-1}\|\Pi^{\pdg-1}(u_h - \hat{u}_h)_t\|_{\partial T}^2 , \\
  | (u_h, \hat u_h, \omega_h) |^2_{U_h,*}
    &:= \sum_{T \in \mesh} \|\dev{\nabla u - \omega_h}\|_T^2 + h^{-1} \| \Pi^{\pdg-1}(u_h - \hat u_h)_t \|_{\partial T}^2.
\end{align}
In \eqref{eq::nhdg_eps}, the terms for $F\in\GammaNT$, where $\hat{u}_h=0$, weakly enforce $u_t=0$ from \eqref{eq::stokes_NT}.
There holds the equivalence (see \cite{gopalakrishnan2021minimal})
\begin{alignat}{2}\label{eq::normequitwo}
  \| (u_h, \hat u_h, \omega_h) \|^2_{U_h} &\sim | (u_h, \hat u_h, \omega_h) |^2_{U_h, *} + d^{-1}\|\div(u_h)\|_0^2
  \quad &&\forall (u_h, \hat u_h, \omega_h) \in U_h.
\end{alignat}

\subsection{Technical results} \label{ssec::intops}

For readability, the technical details of this Section are moved to
the Appendix.

\subsubsection{Interpolation operators} \label{sssec::intops}

A well known interpolation operator $\IVf :
H^2(\mesh, \rr^d) \rightarrow \Vbar_h^f := \Poly^1(\mesh, \rr^d)\cap H^1(\Omega, \rr^d)$ is defined by
  \begin{align}
  (\IVf(u))(p) = \frac{1}{|\chi_p|}\sum_{T\in\chi_p}u_{|T}(p)\quad\forall p\in \verts \label{eq::defintrp},
\end{align}
where $\chi_p$ is the set of all elements that share the vertex $p$,
and $|\chi_p|$ is the number of such elements. 
Bounds for the approximation error of $\IVf$ in $H^1$-like norms are very standard and well known,
and with a Korn inequality for broken $H^1$ spaces like
\begin{align}
  \sum_{T\in\mesh}\|\nabla u\|_T^2 \leq C_K \sum_{T\in\mesh}\|\eps(u)\|_T^2 + \!\!\! \sum_{F\in\facetsI\cup\facetsD}\|\Pi^{R}_F\jump{u}\|^2_{F}\quad\forall u\in H^1(\mesh, \rr^d), \label{eq::globalkorn}
\end{align}
derived in \cite{brenner_korn}, it can easily be bounded by an $\|\cdot\|_{\eps,h}$ like one.
However, as the kernel of $\eps$ is controlled only
by the $\facetsD$ terms,
$C_k$ can degenerate depending on the shape of $\Omega$ and $\GammaD$.
As it would otherwise later on enter into condition number estimates,
the following Lemma \ref{lemma::nc_to_c_nokorn} bounds the approximation error of $\IVf$
independent of $C_K$.
\begin{lemma}\label{lemma::nc_to_c_nokorn}
  There holds
  \begin{align}
      \sum_{T\in\mesh}h^{-2}\|u &- \IVf u\|_T^2 + \|\nabla(u - \IVf u)\|_T^2 
       \lesssim \sum_{T\in\mesh}\|\eps(u)\|_T^2 + \sum_{F\in\facetsI}h^{-1}\|\Pi^R_F\jump{u}\|_F^2
        \quad\forall u\in H^2(\mesh, \rr^d). \label{eq::nc_to_c_nokorn}
  \end{align}
\end{lemma}
\begin{proof}
See Appendix \ref{app::interp}.
\end{proof}
A minor technical detail is our need for an interpolation operator not into
$\Vbar_h^f$ but into $\Vbar_h$.
It can be obtained by simply interpolating into
$\Vbar_h^f$ and then zeroing out degrees of freedom on $\GammaD$ via
\begin{align*}
  \pi_0: \Vbar_h^f \rightarrow \Vbar_h \quad \text{defined by} \quad \pi_0\bar{u}_h(p) =
  \begin{cases}
    0 & p\in\overline{\Gamma}_D \\ 
    \bar{u}_h(p) & \text{else}
  \end{cases}\quad\text{for}\quad p\in\verts.
\end{align*}
\begin{lemma}\label{lemma::nc_to_c_nokorn_withbc}
  For $\IV:H^2(\mesh)\rightarrow\Vbar_h: u\rightarrow \pi_0\IVf u$
  there holds
  \begin{align}\label{eq::nc_to_c_nokorn_withbc}
    \sum_{T\in\mesh}h^{-2}\|u &- \IV u\|_T^2 + \|\nabla(u - \IV u)\|_T^2 
       \lesssim \sum_{T\in\mesh}\|\eps(u)\|_T^2 + \sum_{F\in\facetsI}h^{-1}\|\Pi^R_F\jump{u}\|_F^2
      \quad\forall u\in H^2(\mesh, \rr^d).
  \end{align}
\end{lemma}
\begin{proof}
See Appendix \ref{app::interp}.
\end{proof}

\subsubsection{Trace norms} \label{sssec::trace} 
For $F\in\facets$ and an arbitrary element $T\in\mesh$ with
$F\in\facetsT$ we define for all $\hat{u}\in\Poly^{\pdg}(F,\rr)$
discrete versions of the $H^{1/2}(F, \rr)$ and the $H_{00}^{1/2}(F,
\rr)$-norm for (scalar) HDG spaces as
\begin{align*}
  \|\hat{u}\|_{1,F}^2 &:=
  \inf_{w\in \Poly^{\pdg}(T)}
  \{\|\nabla w\|_T^2 + \|w - \hat{u}\|_{j,F}^2 \}, \quad \textrm{and} \quad
  \|\hat{u}\|_{1,F,0}^2 :=
  \inf_{w\in \Poly^{\pdg}(T)}
  \{\|\nabla w\|_T^2 + \|w - \hat{u}\|_{j,F}^2 + \sum_{\substack{\tilde{F}\in\facetsT\setminus\{F\}}} \|w\|_{j,F}^2\}.
\end{align*}
In \cite{Schb_DDHDG} the authors proved the inverse estimate
\begin{align} \label{eq::scalarinverseest}
  \|\hat{u}\|_{1,F,0}^2 \lesssim (\log{\pdg})^3 \|\hat{u}\|_{1,F}^2
  \quad\forall \hat{u}\in\Poly^{\pdg}(F,\rr) ~\text{such that}~ \Pi^0_F \hat{v} = 0.
\end{align}
A similar estimate can be derived for the
hybrid, vector-valued velocity space $\VUUH$ and norms involving the symmetric gradient,
\begin{align}
  \| (u, \hat{u}) \|_{\eps,F}^2 & :=
  \inf_{\substack{w\in \Poly^{\pdg}(T) \\w_n = u_n\textrm{ on }F}}
  \{\|\eps(w)\|_T^2 + \|\Pi^{\pdg-1}(w - \hat{u})_t\|_{j,F}^2 \} \label{eq::traceF} \\
  \| (u, \hat{u}) \|_{\eps,F,0}^2 & :=
  \inf_{\substack{w\in \Poly^{\pdg}(T) \\w_n = u_n\textrm{ on }F,\\w_n = 0\textrm{ on }\partial T\setminus F}}
    \|\eps(w)\|_T^2 + \|\Pi^{\pdg-1}(w - \hat{u})_t\|_{j,F}^2 + \sum_{\tilde{F}\in\facetsT\setminus\{F\}}\|\Pi^{\pdg-1}w_t\|_{j,\tilde{F}}^2
  \label{eq::traceFZ}
\end{align}
The difference lies not only in the appearance of $\eps$ instead of
$\nabla$ but also, and more importantly, in the fact that, as
$V_h\subseteq H(\div)$, the normal trace is enforced strongly, and one
has to slightly modify the strategy from \cite{Schb_DDHDG}.
\begin{corollary}\label{coro::traceZ}
  For $(u, \hat{u})\in \VUUH$ with $\Pi^R_F (u_nn + \hat{u}_t) = 0$ there holds
  \begin{align}\label{eq::traceZ}
      \| (u, \hat{u}) \|_{\eps,F,0}^2 \lesssim (\log k)^3 \| (u, \hat{u}) \|_{\eps,F}^2.
  \end{align}
\end{corollary}
\begin{proof}
  See Appendix~\ref{app::trace}
\end{proof}

\section{The MCS method}
\label{sec::mcsstokes}
The method considered in this work is based on formulation
\eqref{eq::mixedstressstokes}, where $\omega$ is used as a Lagrange
multiplier to weakly enforce the symmetry constraint
\eqref{eq::mixedstressstokes-c}, see also \cite{MR2336264, MR2629995, Stenberg1988}.
In \cite{lederer2019mass}, a novel variational
formulation of \eqref{eq::mixedstressstokes} without the symmetry constraint was presented where the
velocity and pressure spaces were $H(\div, \Omega)$ and $L^2(\Omega)$
and the stress space for the variable $\sigma$ was defined as
$H(\curl\div):= \{\sigma \in L^2(\Omega, \dd): \div(\sigma) \in
H(\div, \Omega)^* \}$, where the superscript $*$ denotes
the classical dual space.
The variational version of \eqref{eq::mixedstressstokes-b} then became
\begin{align} \label{eq::dualitypair}
  \langle \div(\sigma), v \rangle_{\div} + (\div(v), p) = (f,v) \quad \forall v \in H(\div, \Omega),
\end{align}
where $\langle \cdot, \cdot \rangle_{\div}$ denotes the duality pairing
on $H(\div, \Omega)$.
The authors showed that Finite Element approximation of $\sigma$ in $H(\curl \div)$
demands \textbf{normal-tangential} continuity.
The method described in the following is based on this variational
formulation and in many ways is a variation of previous MCS methods from
\cite{Lederer:2019b,Lederer:2019c,lederer2019mass,gopalakrishnan2021minimal}.
Like the method from \cite{gopalakrishnan2021minimal},
we incorporate the normal-tangential continuity of $\sigma_h$
via a Lagrange multiplier in $\hat V_h$, similar to approaches taken
in hybridized mixed methods for the
Poisson problem, see \cite{MR813687, MR890035, arnold1985mixed,MR3618552}.
For a detailed discussion on this hybridization technique see also
\cite[Section 7.2.2]{brezzi2012mixed}.
The main motivation for breaking the normal-tangential
continuity by hybridization is that it enables local, element-wise elimination,
or static condensation, of all $\Sigma_h$ and $W_h$ dofs.
The resulting, condense, system is the one we actually have to solve,
that is the one we are interested in preconditioning
and will therefore be discussed in great detail in Section \ref{sec:static_cond}.

The {\em hybridized mass conserving mixed stress method with weakly
imposed symmetry} finds $(\sigma_h, (u_h, \hat u_h, \omega_h), p_h)
\in \Sigma_h \times U_h \times Q_h$ such that
\begin{subequations} \label{eq::discrmixedstressstokesweak}
  \begin{alignat}{2}
    -\nu ^{-1}(\sigma_h ,\Stressvarhtest) + b(\tau_h, (u_h, \hat u_h, \omega_h)) & = 0  && \quad\forall \tau_h \in \Sigma_h,  \label{eq::discrmixedstressstokesweak-a}\\
    b(\sigma_h, (v_h, \hat v_h, \eta_h)) + {\nu}{d}^{-1} ( \div (u_h), \div (v_h)) - (\div(v_h), p_h) & = (f, v_h) \label{eq::discrmixedstressstokesweak-b}
    &&\quad \forall (v_h, \hat v_h, \eta_h) \in U_h,  \\
    -(\div (u_h), q_h) &=0  &&\quad\forall q_h \in Q_h, \label{eq::discrmixedstressstokesweak-c}
\end{alignat}                                        
\end{subequations}
with the bilinear form
\begin{align*}
  b(\tau_h, (u_h, \hat u_h, \omega_h)) &:= \sum\limits_{T \in \mesh} \int_T \div(\tau_h) \cdot u_h  \dx - \int_{\partial T} (\tau_h)_{nn} (u_h)_n \ds
                       + \int_T \tau_h : \omega_h \dx - \int_{\partial T} (\tau_h)_{nt} \hat u_h \ds.
\end{align*}
The first two integrals in $b$ can be interpreted as a discrete
version of the duality pair given in \eqref{eq::dualitypair} and the
third weakly enforces the symmetry constraint.
The last terms incorporate the normal-tangential continuity of $\sigma_h$
and the tangential part of \eqref{eq::stokes_N}. Since
\begin{align*}
  \sum\limits_{T \in \mesh} - \int_{\partial T} (\tau_h)_{nt} \hat u_h \ds = \sum\limits_{F \in \facets} \int_{F} \jump{(\tau_h)_{nt}} \hat u_h \ds
\end{align*}
and $\jump{(\tau_h)_{nt}} \in \hat V_h$, testing
\eqref{eq::discrmixedstressstokesweak-b} with all $(0,\hat v_h,0), \hat{v}_h \in \hat
V_h$ results in $\jump{(\sigma_h)_{nt}}=0$ on all $F \in \facetsI$.
On $\GammaD$ and $\GammaNT$ the integrals vanish together with $(\hat{v}_h)_t=0$
and on $\GammaN$ the remaining
integrals weakly incorporate the tangential part of \eqref{eq::stokes_N},
$((-\nu \sigma_h + p_h I)n)_t = (-\nu \sigma_h)_{nt} = 0$.
More details on boundary conditions in all possible combinations can be found in
\cite{lederer2019mass}.

The term ${\nu}{d}^{-1} ( \div (u_h), \div (v_h))$ was added to
guarantee inf-sup stability of the diffusive sub problem
\eqref{eq::mcs_epseps} defined below. However, since the solution
$u_h$ is exactly divergence-free (by
\eqref{eq::discrmixedstressstokesweak-c} and $\div (V_h) = Q_h$), the
added term is consistent.
 Finally note that we did not include the deviator in the discrete
formulation (as compared to \eqref{eq::mixedstressstokes-a}) since
functions $\tau_h \in \Sigma_h$ are elements of $\dd$ and so
$\dev{\tau_h} = \tau_h$.

For the definition of the preconditioner derived later, we define the
sub problem: Find $(\sigma_h, (u_h, \hat u_h, \omega_h)) \in \Sigma_h
\times U_h $ such that
  \begin{align}\label{eq::mcs_epseps}
    \mathcal{K}((\sigma_h, (u_h, \hat u_h, \omega_h)),&(\tau_h, (v_h, \hat v_h, \eta_h)))
    = (f, v_h) \quad \forall (\tau_h, (v_h, \hat v_h, \eta_h)) \in \Sigma_h \times U_h,
  \end{align}  
  with
  \begin{align*}
    \mathcal{K}((\sigma_h, &(u_h, \hat u_h, \omega_h)),(\tau_h, (v_h, \hat v_h, \eta_h))) = 
    -\nu^{-1}(\sigma_h ,\tau_h) +  b(\tau_h, (u_h, \hat u_h, \omega_h)) + b(\sigma_h, (v_h, \hat v_h, \eta_h))+ {\nu}{d}^{-1} ( \div (u_h), \div (v_h)).
  \end{align*}
  Note that with the added $\nu d^{-1}(\div({u_h}),\div({v_h}))$ term, equation \eqref{eq::mcs_epseps} reads as a discrete
  variational formulation of the elliptic problem $-\div(\nu \eps(u))
  = f$.

\begin{remark}
Note that the MCS method here, like the one from \cite{Lederer:2019c} it is most closely related to,
is only stable for $\pdg\geq 2$.
A stable minimal order MCS method with $\pdg=1$ was introduced in
\cite{gopalakrishnan2021minimal} and we will return to it in Section \ref{sec::lorder}.
\end{remark}

\subsection{Stability analysis}
In the following we summarize the stability results for the discrete method defined above.
We only prove solvability of \eqref{eq::mcs_epseps}, all other results follow with
the same techniques and steps as in
\cite{Lederer:2019b,Lederer:2019c,lederer2019mass,gopalakrishnan2021minimal}.
Lemma~\ref{lemma::lbb}, which can, just as Lemma \ref{lemma::continuity}, be found in 
the stated literature, is an inf-sup stability result for the constraint
given by the bilinear form $b$.
It is posed in the semi-norm $|\cdot|_{U_h,*}$ as,
since all elements in $\Sigma_h$ are trace-free, the divergence of
functions in $U_h$ can not be controlled.
Theorem~\ref{th::infsupsubprob} states that
with the addition of the term
$\nu d^{-1} (\div(u_h), \div(v_h))$ in \eqref{eq::discrmixedstressstokesweak-b}
we can switch to the proper norm $\| \cdot \|_{U_h}$
and \eqref{eq::mcs_epseps} is solvable independently of the divergence constraint. 
Finally, Corollary \ref{cor::infsupstokes}, which is again already proven in the
literature, provides solvability of
\eqref{eq::discrmixedstressstokesweak} including the divergence
constraint.

\begin{lemma}\label{lemma::continuity}
  There hold the continuity estimates
  \begin{align*}
    \nu^{-1}(\sigma_h, \tau_h) &\le \nu^{-1} \| \sigma_h\|_0 \| \tau_h\|_0  &&\forall \tau_h,\sigma_h \in \Sigma_h,\\
    b(\tau_h, (u_h,\hat u_h, \omega_h)) & \lesssim \| \tau_h\|_0 \| (u_h, \hat u_h, \omega_h) \|_{U_h}  &&\forall (\tau_h, (u_h,\hat u_h, \omega_h)) \in \Sigma_h \times U_h,\\
    (\div(u_h), q_h) &\lesssim \| (u_h, \hat u_h, \omega_h) \|_{U_h} \|q_h\|_0 &&\forall (q_h, (u_h,\hat u_h, \omega_h)) \in Q_h \times U_h,\\
    {\nu}{d}^{-1} ( \div (u_h), \div (v_h)) &\le \nu d^{-1} \| \div(u_h)\|_0  \| \div(v_h) \|_0  &&\forall u_h, v_h \in V_h.
  \end{align*}
\end{lemma}

\begin{lemma} \label{lemma::lbb}
Let $(v_h, \hat v_h, \eta_h) \in U_h$ be arbitrary. There exists a $\sigma_h \in \Sigma_h$ such that 
\begin{align*}
  b(\sigma_h, (v_h,\hat v_h, \eta_h)) \gtrsim | (v_h, \hat v_h, \eta_h) |^2_{U_h,*},  \quad \textrm{and} \quad \| \sigma_h\|_0 \lesssim | (v_h, \hat v_h, \eta_h) |_{U_h,*}.
  \end{align*}
\end{lemma}

\begin{theorem} \label{th::infsupsubprob}
  Let $(\tau_h, (v_h, \hat v_h, \eta_h)) \in \Sigma_h \times U_h$ be
  arbitrary, there holds the inf-sup stability
  \begin{align*}
    \sup_{\substack{\sigma_h \in \Sigma_h \\
    (u_h, \hat u_h, \eta_h) \in U_h}} &\frac{\mathcal{K}((\sigma_h, (u_h, \hat u_h, \omega_h)),(\tau_h, (v_h, \hat v_h, \eta_h)))}{{\nu}^{-1/2} \|\sigma_h\|_0 + {\nu}^{1/2}\| (u_h, \hat u_h, \omega_h) \|_{U_h}} 
                                         \gtrsim {\nu}^{-1/2}\|\tau_h\|_0 + {\nu}^{1/2}\| (v_h, \hat v_h, \eta_h) \|_{U_h}.
  \end{align*}
\end{theorem}
\begin{proof}
  This follows with standard techniques, i.e. using Lemma
  \ref{lemma::lbb}, Young's and Cauchy Schwarz's inequality and the
  norm equivalence \eqref{eq::normequitwo}.
\end{proof}

\begin{corollary} \label{cor::infsupstokes}
   Let $(\tau_h, (v_h, \hat v_h, \eta_h), q_h) \in \Sigma_h \times U_h \in Q_h$ be
   arbitrary, there holds the inf-sup stability
  \begin{align*}
    \sup_{\substack{\sigma_h \in \Sigma_h \\
    (u_h, \hat u_h, \eta_h) \in U_h \\ p_h \in Q_h}} &\frac{\mathcal{K}((\sigma_h, (u_h, \hat u_h, \omega_h)),(\tau_h, (v_h, \hat v_h, \eta_h))) + (\div (u_h), q_h) + (\div (v_h), p_h)}{{\nu}^{-1/2}( \|\sigma_h\|_0 + \| p_h\|_0) + {\nu}^{1/2}\| (u_h, \hat u_h, \omega_h) \|_{U_h}} \\
                                                     &\qquad\qquad \qquad \gtrsim ({\nu}^{-1/2}\|\tau_h\|_0 + \| q_h\|_0) + {\nu}^{1/2}\| (v_h, \hat v_h, \eta_h) \|_{U_h}.
  \end{align*}  
\end{corollary}

\subsection{Static condensation of local variables}
\label{sec:static_cond}

We now discuss the structure of the Finite Element matrix directly
obtained from the MCS method \eqref{eq::discrmixedstressstokesweak}
and that of various Schur complements thereof.
Writing $\phi^\sigma,\phi^u,\phi^{\hat u},\phi^\omega$ and $\phi^p$ for the basis
functions of $\Sigma_h,V_h, \hat V_h, W_h$ and $Q_h$ respectively
and, complying with the notation for the Galerkin isomorphism introduced
in Section \ref{sec::notation},
$\vu$ for the coefficients of $u_h$ with respect to the basis
given by $\phi^u$, etc.,
\eqref{eq::discrmixedstressstokesweak} in matrix form is
\begin{align}\label{eq::bigfemSPmat}
  \begin{pmatrix}
    \begin{pmatrix}
      -\mM_{\vsig \vsig} & \mB_{ \vomega \vsig}^\trans \\
      \mB_{ \vomega \vsig} & \mZ \\
    \end{pmatrix} &
    \begin{pmatrix}
      \mB_{ \vu \vsig}^\trans & \mB_{ \vuh \vsig}^\trans \\
      \mZ & \mZ \\
    \end{pmatrix}&
    \begin{pmatrix}
      \mZ \\
      \mZ \\
    \end{pmatrix}\\
    \begin{pmatrix}
      \mB_{  \vu \vsig} &\mZ \\
      \mB_{ \vuh \vsig} & \mZ \\
     \end{pmatrix} &
     \begin{pmatrix}
      \mA^{\div}_{\vu \vu} & \mZ  \\
      \mZ&\mZ\\
    \end{pmatrix}&
    \begin{pmatrix}
      \mB_{\vp\vu}^T \\
      \mZ \\
     \end{pmatrix}\\
     \begin{pmatrix}
       \mZ&\mZ\\
     \end{pmatrix} &
     \begin{pmatrix}
       \mB_{\vp\vu} & \mZ \\
     \end{pmatrix} &
       \mZ
   \end{pmatrix}
  \begin{pmatrix}
    \begin{pmatrix}
      \vsig \\
      \vomega\\
    \end{pmatrix}\\
    \begin{pmatrix}
      \vu\\
      \vuh\\
    \end{pmatrix}\\
    \vp
  \end{pmatrix}
  =
  \begin{pmatrix}
    \begin{pmatrix}
      \mZ\\\mZ\\
    \end{pmatrix}\\
    \begin{pmatrix}
      \vf \\\mZ\\
    \end{pmatrix}\\
    \mZ\\
  \end{pmatrix}.
\end{align}
The right hand side vector $\vf$ is given by
$\vf_i = (f, \phi^u_i)$ and the system matrix with
\begin{align*}
  (\mM_{\vsig \vsig})_{ij} &= \nu^{-1}(\phi_i^\sigma,\phi_j^\sigma),&
  ( \mB_{\vuh \vsig} )_{ij} &= \sum\limits_T -\int_{\partial T} (\phi_j^\sigma)_{nt} (\phi_i^{\hat u})_t \ds,\\
  ( \mB_{\vomega \vsig} )_{ij} &= \sum\limits_T \int_{\partial T} \phi_j^\sigma:\phi_i^{\omega} \dx,&
  ( \mA^{\div}_{\vu \vu} )_{ij} &= \nu d^{-1} (\div{(\phi_i^u)},\div (\phi_j^{u})), \\
  ( \mB_{\vp\vu} )_{ij} &= (\div{(\phi^u_j)}, \phi^p_i), &
  (\mB_{\vu \vsig})_{ij} &= \sum\limits_T \int_T \div(\phi_j^\sigma) \phi_i^u \dx - \int_{\partial T} (\phi_j^\sigma)_{nn} (\phi_i^u)_n \ds.
\end{align*}
is a saddle point matrix with Lagrange multipliers
$\vomega, \vu, \vuh$ and $\vp$ enforcing
\eqref{eq::mixedstressstokes-c},
\eqref{eq::mixedstressstokes-a},
the $nt$-continuity of $\sigma$, and
\eqref{eq::mixedstressstokes-d} respectively.
\paragraph{Static condensation of $\vsig, \vw$}
The diagonal block for $\vsig, \vw$ does not couple
with the incompressibility constraint and,
thanks to the introduction of $\hat{u}_h$ as additional
multiplier, is block diagonal.
It is also invertible since every block
represents the simple projection problem of finding
$(\sigma_h^T, \omega_h^T) \in \Sigma_h(T) \times W_h(T)$
for some $T \in \mesh$ such that
\begin{subequations}\label{eq::projection}
\begin{alignat}{2}
  -\frac{1}{\nu}(\sigma_h^T,\tau_h)_T + (\tau_h, \omega_h^T)_T &= g_T(\tau_h) &\quad \forall \tau_h \in \Sigma_h(T)\label{eq::projectionone}, \\
  (\sigma^T_h, \eta_h)_T &= 0 &\quad \forall \eta_h \in W_h(T),\label{eq::projectiontwo}
\end{alignat}
\end{subequations}
where $\Sigma_h(T), W_h(T)$ are the restrictions of the
corresponding (discontinuous) global spaces to $T$ and $g_T$ is some
right hand side.
Standard arguments and the Brezzi theorem prove
that \eqref{eq::projection} is inf-sup stable, that is, writing
\begin{align*}
  \mM :=
  \begin{pmatrix}
    -\mM_{\vsig \vsig} & \mB_{ \vomega \vsig}^\trans \\
    \mB_{ \vomega \vsig} & \mZ \\
  \end{pmatrix},
  \quad \mB_{\vsig} :=
  \begin{pmatrix}
    \mB_{ \vu \vsig} &\mZ \\
    \mB_{ \vuh \vsig} & \mZ \\
     \end{pmatrix},
  \quad \mA^{\div} :=
  \begin{pmatrix}
       \mA^{\div}_{\vu \vu} & \mZ  \\
       \mZ&\mZ\\
     \end{pmatrix},
\end{align*}
$\mM$ is invertible and the Schur complement $\SsM := \mA^{\div} -
\mB_{\vsig} \mM^{-1}\mB_{\vsig}^\trans$ is well defined and, as $\mM$
is block diagonal, can be computed element-wise. Eliminating
$\vsig,\vomega$ from \eqref{eq::bigfemSPmat} in this way leaves us
with the system
\begin{align}\label{eq::femSPmat}
  \mK
  \begin{pmatrix}
    \begin{pmatrix}
      \vu\\
      \vuh\\
    \end{pmatrix}\\
    \vp
  \end{pmatrix}
  :=
  \begin{pmatrix}
    \SsM & \mB^T \\
    \mB & \mZ\\
  \end{pmatrix}
  \begin{pmatrix}
    \begin{pmatrix}
      \vu\\
      \vuh\\
    \end{pmatrix}\\
    \vp
  \end{pmatrix}
  =
  \begin{pmatrix}
    \begin{pmatrix}
      \vf \\\mZ\\
    \end{pmatrix}\\
    \mZ\\
  \end{pmatrix}.
\end{align}
The symmetry of $\mK$ is obvious and in the next Lemma \ref{lemma::SPD}
we show that the upper left block $\SsM$ is also
positive definite and we are now in the very standard setting of a
saddle point problem with symmetric and positive definite (SPD) ``$A$-block''.
The velocity unknowns $u_h,\hat{u}_h$ move to the position of primal variables,
while the pressure $p_h$ remains the Lagrange parameter for the divergence constraint.
After solving \eqref{eq::femSPmat} to get $u_h, \hat{u}_h, p_h$, we can recover $\sigma_h$
and $\omega_h$ by solving the local problems \eqref{eq::projection}.
\begin{lemma} \label{lemma::SPD}
  The Schur complement $\SsM$ is symmetric positive definite and with
  \begin{align*}
    \VUUHC:=\{(v_h, \hat{v}_h)\in \VUUH : v_h\in H^1(\Omega, \rr^d),
    \Pi^{\pdg-1}_F(u_h - \hat{u}_h)_t=0~\forall F\in\facets\}
  \end{align*} 
  there holds
  \begin{align}
    \nu \| (u_h, \hat{u}_h) \|^2_{\eps,h} \lesssim \|(u_h&, \hat{u}_h)\|_{\Ss}^2 \leq \nu \| (u_h, \hat{u}_h) \|^2_{\eps,h}
    \quad \forall (u_h,\hat{u}_h) \in \VUUH, \label{eq::schur_SPD} \\
    \|(u_h, \hat{u}_h)\|_{\Ss}^2 &= \nu\|\eps(u_h)\|_0^2 \quad\forall (u_h,\hat{u}_h)\in\VUUHC. \label{eq::schur_SPD_cont}
  \end{align}
\end{lemma}
\begin{proof}
  Let $(u_h,\hat{u}_h) \in \VUUH$ be arbitrary and set
  $(\vsig, \vomega) := -\mM^{-1} \mB_{\vsig}^T (\vu, \vuh)$, i.e. the local functions
  $(\sigma_{h|T}, \omega_{h|T}) := (\sigma^T_h ,\omega^T_h)$ are the
  solution of \eqref{eq::projection} with right hand side
  \begin{align} \label{eq::projrighthandside}
    g_T(\tau_h) := -b(\tau_h, (u_h, \hat u_h, 0)) = \int_T \tau_h : \nabla u_h  \dx - \int_{\partial T} (\tau_h)_{nt} (u_h - \hat u_h)_t \ds,
  \end{align}
  where we used an element-wise integration by parts for $b$.
  From \eqref{eq::projectiontwo} we see
  $(\mB_{\vomega \vsig} \vsig,\vomega ) = \int_{\Omega}\sigma_h:\omega_h \dx= 0$
  and there holds
  \begin{align*}
    \|(u_h, \hat u_h)\|_{\Ss}^2
    &= (\vsig, \vomega, \vu, \vuh)
      \begin{pmatrix}
        \mM & \mB_{\vsig}^\trans \\
        \mB_{\vsig} & \mA^{\div}\\
      \end{pmatrix}
    (\vsig, \vomega, \vu, \vuh)^\trans \\
    &= (\mM_{\vsig \vsig} \vsig, \vsig) + (\mA^{\div} (\vu, \vuh),(\vu, \vuh) ).
  \end{align*}
  With $ (\mA^{\div} (\vu, \vuh),(\vu, \vuh) ) = \nu d^{-1}\|\div(u_h)\|_0^2$ this gives
  \begin{align}\label{eq::schurnormmat}
      \|(u_h, \hat u_h)\|_{\Ss}^2 = \nu^{-1}\|\sigma_h\|_0^2 +  \nu d^{-1}\| \div(u_h)\|_0^2.
  \end{align}
  We now insert \eqref{eq::projrighthandside} into \eqref{eq::projectionone} and test with $\tau_h=\sigma_h^T$.
  The term $(\tau_h,\omega_h^T)_T=0$ drops out due to \eqref{eq::projectiontwo} and we see that $\forall T\in\mesh$
  \begin{align*}
    \nu^{-1}\| \sigma_h \|_T^2 = g_T(\sigma_h).
  \end{align*}
  We can use \eqref{eq::projectiontwo} again to see $(\sigma_h,\nabla u_h)_T = (\sigma_h,\dev{\eps(u_h)})_T$ ,
  as $\trace{\sigma_h}=0$ is built into $\Sigma_h(T)$ and get
  \begin{align*}
    \nu^{-1}\| \sigma_h \|_T
    & \leq \frac{|(\nabla u_h, \sigma_h)_T| + |((\sigma_h)_{nt},(u_h-\hat{u}_h)_t)_{\partial T}|}{\|\sigma_h\|_T} \\
    & = \frac{|(\dev{\eps(u_h)}, \sigma_h)_T| + \sum_{F\in\facetsT}|((\sigma_h)_{nt},(u_h-\hat{u}_h)_t)_{F}|}{\|\sigma_h\|_T} \\
    & \leq \|\dev{\eps(u_h)}\|_T + \sum_{F\in\facetsT}\sup_{\tau_h\in\Sigma_h(T)} \frac{((\tau_h)_{nt},(u_h-\hat{u}_h)_t)_{F}}{\|\tau_h\|_T} \\
    & \leq \|\dev{\eps(u_h)}\|_T + \sum_{F\in\facetsT}\|\Pi^{\pdg-1}(u_h-\hat{u}_h)_t\|_{j,F}.
  \end{align*}
  Thus, with \eqref{eq::schurnormmat}
  \begin{align*}
    \|(u_h,\hat{u}_h)\|_A^2
    & \leq \nu\sum_{T\in\mesh} \Big(\|\dev{\eps(u_h)}\|_T^2 + d^{-1}\|\div(u_h)\|_T^2 
     + \sum_{F\in\facetsT}\|\Pi^{\pdg-1}(u_h-\hat{u}_h)_t\|_{j,F}^2\Big) \\
    & = \nu\sum_{T\in\mesh}\Big( \|\eps(u_h)\|_T^2 + \sum_{F\in\facetsT}\|\Pi^{\pdg-1}(u_h-\hat{u}_h)_t\|_{j,F}^2\Big) 
     = \nu\|(u_h,\hat{u}_h)\|_{\eps,h}^2.
  \end{align*}
  It remains to prove the other direction.
  By Lemma \ref{lemma::lbb} there exists a $\tau_h \in \Sigma_h$ with
  $\| \tau \|_0 \lesssim | u_h, \hat u_h, \omega_h|_{U_h,*}$ such
  that, once again inserting \eqref{eq::projrighthandside} into \eqref{eq::projectionone},
  we see that
  \begin{align*}
    | (u_h, \hat u_h, \omega_h) |_{U_h,*} \lesssim \frac{b(\tau_h, (u_h, \hat u_h, \omega_h))}{\|\tau_h\|_0} 
     = \sum\limits_{T\in\mesh} \frac{\nu^{-1}(\sigma_h, \tau_h)_T}{\|\tau_h\|_0} \leq \nu^{-1}\|\sigma_h\|_0
  \end{align*}
  and therefore there holds
  \begin{align*}
    \nu\|(u_h,\hat{u}_h)\|_{\eps,h}^2 &\lesssim \nu|(u_h, \hat u_h, \omega_h) |^2_{U_h,*} +  \nu d^{-1} \|\div(u_h)\|_{0}^2 
                                      \lesssim \nu^{-1}\|\sigma_h\|_0^2 +  \nu d^{-1} \|\div(u_h)\|_{0}^2 
                                       = \|(u_h,\hat{u}_h)\|_A^2.
  \end{align*}
  Finally, for $(u_h, \hat u_h) \in \VUUHC$ we have $\Pi^{\pdg-1}_F(u_h - \hat{u}_h)_t=0~\forall F\in\facets$
  and the distributional terms in \eqref{eq::projrighthandside} vanish.
  The solutions of \eqref{eq::projection} are then simply given by
  $\sigma_h = - \nu \dev{\eps(u_h)}$ and $\omega_h = \kappa(\curl(u_h))$
  and \eqref{eq::schurnormmat} states  
  \begin{align*}
      \|(u_h, \hat u_h)\|_{\Ss}^2 &= \nu\| \dev{\eps(u_h)}\|_0^2 + \nu d^{-1} \|\div(u_h)\|_0^2 = \nu \| \eps(u_h)\|_0^2.
  \end{align*}
\end{proof}

\begin{remark}
In $\Ss$, we have a discretization of $\div(\nu\eps(u))$ with degrees
of freedom $u_h$ and $\hat{u}_h$ only. This is less reminiscent of a
mixed method like MCS than of a HDG method and it is interesting to
further elaborate on the relationship between the MCS method and DG
and HDG methods. In general, DG and HDG methods require a stabilizing
term to assure solvability. An example is the well known interior
penalty method where the $L^2$-norm of jumps, $\alpha
\frac{\pdg^2}{h}\|u_h - \hat{u}_h\|_F^2$ for $F\in\facets$ with some
sufficiently large $\alpha$ is used. Any dependence on such a
parameter is avoided here, however this is not an unique feature of
the MCS method. Other DG and HDG methods that also avoid this
parameter feature a lifting $\sigma_h$ of the jump similar to
\eqref{eq::jFjumpN} instead of its $L^2$ norm, see \cite{MR2684358}.
That lifting has to be explicitly computed and is then condense out. A
final class of DG methods, for example the one in \cite{Bassi1997},
see also \cite{MR2684358, cockburn2009unified}, features a
simultaneous lifting of the jump and the fluxes. This is
similar to what happens here, where $\sigma_h$ both approximates the
flux $-\nu\eps(u)$ and \emph{automatically and canonically} stabilizes
the condense system through its interaction with the tangential jumps.
\end{remark}

\paragraph{Static condensation of high order velocity functions}
The base functions of $V_h$ can be split into
two different types, see \cite{Beuchler2012}.
We write $\phi^{u,\circ}$ for the high order ``element bubble'' base functions
whose support is entirely within some element $T\in\mesh$ and whose
normal trace on $\partial T$ vanishes.
The span of these base functions is denoted by $V_h^{\circ}$ and
we write $\VUUHI := V_h^{\circ}\times\{0\}\subseteq \VUUH$.
The remaining base functions $\phi^{u,\partial}$ of $V_h$ have support
entirely within the patch of some facet $F$ and their
normal trace on all other facets in the patch vanishes.
As the supports of different $\phi^{u,\circ}$ do not overlap, in
\begin{align*}
\SsM = 
  \begin{pmatrix}
    \SsM_{\circ\circ} & \SsM_{\circ\partial} & \SsM_{\circ\vuh} \\
    \SsM_{\partial\circ} & \SsM_{\partial\partial} & \SsM_{\partial\vuh} \\
    \SsM_{\vuh\circ} & \SsM_{\vuh\partial} & \SsM_{\vuh\vuh} \\
  \end{pmatrix}
\end{align*}
the upper left block $\SsM_{\circ\circ}$ is block diagonal and invertible. 
This lets us form a second, ``double'' Schur complement
\begin{align*}
  \SsDM:=
  \begin{pmatrix}
    \SsM_{\partial\partial} & \SsM_{\partial\vuh} \\
    \SsM_{\vuh\partial} & \SsM_{\vuh\vuh} \\
  \end{pmatrix}
  -
  \begin{pmatrix}
    \SsM_{\partial\circ} \\
    \SsM_{\vuh\circ}
  \end{pmatrix}
  \SsM_{\circ\circ}^{-1}
  \begin{pmatrix}
    \SsM_{\circ\partial} & \SsM_{\circ\vuh}
  \end{pmatrix}.
\end{align*}
In the bigger system \eqref{eq::femSPmat}, \emph{all}
$V_h$ degrees of freedom couple with the  divergence constraint and
we cannot perform this static condensation independently of the pressure variables.
However, for higher order problems,
implementing multiplication with $\SsM$ via the exact factorization
\begin{align}\label{eq::factorSsigma}
  \SsM&
  =
  \begin{pmatrix}
    \mI&\SsM_{\partial\circ}\SsM_{\circ\circ}^{-1}\\
    \mZ&\mI\\
  \end{pmatrix}
  \begin{pmatrix}
    \SsDM & \mZ \\
    \mZ&\SsM_{\circ\circ}\\
  \end{pmatrix}
  \begin{pmatrix}
    \mI&\mZ\\
    \SsM_{\circ\circ}^{-1}\SsM_{\circ\partial}&\mI\\
  \end{pmatrix},
\end{align}
is still advantageous.
Both the left and right factors as well as $\SsM_{ \circ\circ}$
are block diagonal and only $\SsDM$ instead of the larger $\SsM$ needs to be assembled
as a proper sparse matrix.
We will revisit the idea of also preconditioning $\SsM$ via this factorization
in Section \ref{ssec::schur_schur}.

Splitting the coordinate vector $\vu$ of the $V_h$ component of $(u_h,\hat{u}_h)\in \VUUH$
into $\vu_{\circ}$ and $\vu_{\partial}$, the norm induced by $\SsDM$ on $(\vu_{\partial}, \vuh)$ is
\begin{align} \label{eq::doubleschurinf} 
  \|(\vu_{\partial}, \vuh)\|_{\SsDM} = \inf_{\vof{v}_{\circ}} \|(\vu_{\circ}+\vof{v}_{\circ},\vu_{\partial}, \vuh)\|_{\SsM}
  = \inf_{(v_h, \hat{v}_h)\in \VUUHI} \| (u_h + v_h, \hat{u}_h + \hat{v}_h) \|_{\Ss}. 
\end{align}
That is, the norm induced by $\SsDM$ is just the one induced by $\SsM$ on the
energy minimal extension to $\VUUHI$ dofs.
The lifting operator, or (discrete) harmonic extension, $\mathcal{H}:\VUUH\rightarrow\VUUH$ maps $(u_h,\hat{u}_h)$,
to the minimizer in \eqref{eq::doubleschurinf}:
\begin{align*}
  \mathcal{H}(u_h, \hat{u}_h) = \argmin_{(v_h, \hat{v}_h)\in \VUUHI} \| (u_h+v_h, \hat{u}_h + \hat{v}_h) \|_{\SsM}.
\end{align*}
Equivalently, writing $(\vof{w}, \vof{\hat{w}}) \GalEq (w_h,\hat{w}_h) := \mathcal{H}(u_h, \hat{u}_h)$,
$\mathcal{H}$ is defined by
\begin{align}\label{eq::Hcoeffs}
  \vw_{\partial} = \vu_{\partial}, \quad\quad \vof{\hat{w}} = \vuh \quad\textrm{and}\quad
  \vw_{\circ} = - \SsM_{\circ}^{-1} (\SsM_{\circ\partial}\vu_{\partial} + \SsM_{\circ \vuh }\vuh).
\end{align}
The range of $\mathcal{H}$ is
\begin{align}\label{def::vuuhh}
  \VUUHH
  :=
  \mathcal{H}(\VUUH)
  =
  \{(u_h, \hat{u}_h)\in \VUUH :
  \SsM_{\circ\circ} \vu_{\circ}
  + \SsM_{\circ\partial} \vu_{\partial}
  + \SsM_{\circ\vuh} \vuh
   = 0 \},
\end{align}
and for such ``discrete harmonic'' or ``lifted'' functions $(u_h,\hat{u}_h)\in\VUUHH$ there holds
\begin{align}\label{eq::schuronharmfkt}
    \| (u_h, \hat{u}_h) \|_{\Ss} = \|(\vu_{\circ},\vu_{\partial}, \vuh)\|_{\SsM} = \|(\vu_{\partial}, \vuh)\|_{\SsDM}.
\end{align}
Here we encounter a slight complication of notation:
Per default, $(u_h, \hat{u}_h)\in\VUUHH$ is associated with it's coordinate vector $(\vu^{\circ},\vu^{\partial},\vuh)$,
but $\SsDM$ only takes the $(\vu^{\partial},\vuh)$ coordinates
(which determine $\vu^{\circ}$ according to \eqref{def::vuuhh}).
The space $\VUUHH$ is spanned by lifted, discrete harmonic, basis functions,
\begin{align*}
  \VUUHH =
  \textrm{span}\{\mathcal{H}(\phi^{u,\partial},0) : \phi^{u,\partial}\in\Phi^{u,\partial}\}
  + \textrm{span}\{\mathcal{H}(0,\phi^{\hat{u}}) : \phi^{\hat{u}}\in\Phi^{\hat{u}}\},
\end{align*}
where $\Phi^{u,\partial}$ is the set of all $\phi^{u,\partial}$
and $\Phi^{\hat{u}}$ the one of all $\phi^{\hat{u}}$ basis functions.
The induced Galerkin Isomorphism $\GalEqD$ defines
the natural operator $\SsD$ associated with $\SsDM$
and identifies $(u_h,\hat{u}_h)\in\VUUHH$ with  $(\vu^{\partial},\vuh)$.
Where there is potential for confusion we explicitly write
$(u_h,\hat{u}_h)\GalEqD(\vu^{\partial},\vuh)$ in contrast to
$(u_h,\hat{u}_h)\GalEq(\vu^{\circ},\vu^{\partial},\vuh)$.

Analogously, we define the Schur complement like norm
\begin{align}\label{eq::epshschur}
  \|(u_h,\hat{u}_h)\|_{\eps,h,\partial} := \inf_{(v_h,\hat{v}_h)\in\VUUHI} \|(u_h+v_h,\hat{u}_h+\hat{v}_h)\|_{\eps,h}
\end{align}
and the associated lifting operator $\mathcal{H}_{\eps} : \VUUH \rightarrow \VUUH$ such that
\begin{align*}
  \|(u_h,\hat{u}_h)\|_{\eps,h,\partial} = \|\mathcal{H}_{\eps}(u_h,\hat{u}_h)\|_{\eps,h,} \quad\forall(u_h,\hat{u}_h)\in\VUUH.
\end{align*}
Note that both $\|\cdot\|_{\SsD}$ and  $\|\cdot\|_{\eps,h,\partial}$
can be defined on the entirety of $\VUUH$ but are only semi-norms
as they vanish on $\ker{\mathcal{H}} = \ker{\mathcal{H}}_{\eps} = \VUUHI$.
Restricted to $\VUUHH$ they are proper norms and equivalent.

\begin{corollary} \label{coro::schur_SPD}
  There holds
  \begin{align}
    \gamma^{-1}\nu\|(u_h,\hat{u}_h)\|^2_{\eps,h,\partial} \leq \|(u_h,\hat{u}_h)\|^2_{\SsD} \leq \nu\|(u_h,\hat{u}_h)\|^2_{\eps,h,\partial}\quad\forall (u_h,\hat{u}_h)\in\VUUHH\label{eq::schur_schur_SPD}
  \end{align}
  for some $\gamma>0$.
\end{corollary}
\begin{proof}
  Follows immediately from Lemma \ref{lemma::SPD}.
\end{proof}
Unlike the constant in the lower bound in \eqref{eq::schur_SPD},
$\gamma$ later on directly enters into condition number estimates
and it is important to talk about whether, or how, it depends on the polynomial degree $\pdg$.
Numerical experiments on the unit tetrahedron suggest
$\gamma = \mathcal{O}(1)$ or possibly $\gamma = \mathcal{O}(\log(\pdg)^l)$ with some moderate $l>0$.
We have not pursued further a rigorous proof of this fact.
Such a proof would essentially require a $\pdg$-explicit version of Lemma \ref{lemma::lbb} in $\VUUHH$.



\section{Preconditioning framework} \label{sec::pc_frame}

A final  Schur complement can be formed with respect to
the pressure unknowns, however this involves the inverse $\SsM^{-1}$.
With the resulting (negative) pressure Schur complement $\mSp:=\mB\SsM^{-1}\mB^T$,
we have the exact factorization
\begin{align}\label{eq::Kfactor}
\mK =
\left(\begin{matrix}
      \mI &    \mZ \\
      \mB \SsM^{-1} & \mI
    \end{matrix}\right)
  \left(\begin{matrix}
      \SsM & \mat{0} \\
      \mat{0} & -\mSp
    \end{matrix}\right)
  \left(\begin{matrix}
      \mI & \SsM^{-1}\mB^T \\
      \mat{0} & \mI
    \end{matrix}\right)
\end{align}
for the saddle point matrix $\mK$.
Solving \eqref{eq::femSPmat}
could in principle be reduced to solving separate problems
for the pressure and velocity.
While this is not feasible due to the appearance of $\SsM^{-1}$
in the pressure Schur complement, this line of thought still
takes a prominent role in common preconditioning techniques for $\mK$
based on separate preconditioners $\SsMH$ for
$\SsM$ and $\mSpH$ for $\mSp$.
See \cite{benzi_golub_liesen_2005} and the references therein for an overview of such methods.
Motivated by \eqref{eq::Kfactor}, here we use
\begin{align}\label{eq::mhatinvfactor}
\mKH^{-1} := 
  \left(\begin{matrix}
      \mI & -\SsMH\mB^T \\
      \mat{0} & \mI
    \end{matrix}\right)
  \left(\begin{matrix}
      \SsMH^{-1} & \mat{0} \\
      \mat{0} & \mSpH^{-1}
    \end{matrix}\right)
  \left(\begin{matrix}
      \mI &    \mat{0} \\
      -\mB\SsMH^{-1} & \mI
    \end{matrix}\right).
\end{align}
Note that unlike suggested by \eqref{eq::mhatinvfactor},
the operation $\vx\mapsto \mKH^{-1}\vx$ can be implemented
such that it requires only two applications of $\SsMH^{-1}$ instead of three.
A rigorous analysis of $\mKH$ for the generic saddle point case
as well as a number of other, similar, preconditioners built
from $\SsMH$ and $\mSpH$ can be found in \cite{zulehner_sp_ua}.

\subsection{Pressure Schur preconditioner}

From the standard Stokes-LBB condition on $\VUUH$ using the norm $\| \cdot \|_{\eps,h}$,
see for example in \cite{gopalakrishnan2021minimal},
and the equivalence result Lemma~\ref{lemma::SPD}
we can conclude the MCS Stokes-LBB condition
\begin{align*}
  \sup\limits_{(v_h, \hat v_h) \in \VUUH} \frac{(\div(v_h), q_h)}{\|(v_h,\hat v_h)\|_{\Ss}} \ge \gamma_L \| q_h \|_0 \quad \forall q_h \in Q_h.
\end{align*}
It is generally well known that given this LBB condition, $\mSp$
is equivalent to the scaled mass matrix $(\mM_p \vp, \vq) := \nu^{-1}
(p_h,q_h)_0$ for $p_h,q_h\in Q_h$, see \cite{Verfrth1984ACC,
Wathen1993FastIS}. As $Q_h$ is completely discontinuous across
elements in the MCS discretization, inverting the
block-diagonal matrix $\mM_p$ is feasible and we use
$\mSpH:=\mM_p$.

Note that for norms similar to the ones used here, $\gamma_L$ was proven
to be independent of $\pdg$ in two dimensions in \cite{LedererSchoeberl2016}
and numerical experiments performed in the same work strongly suggest
that the independence also holds in three dimensions.

\subsection{Auxiliary space preconditioning} \label{ssec::aux_frame}
We give the fictitious space Lemma \ref{lemma::fictspace} below
in the compact form it takes, for example, in \cite[Theorem 6.3]{xu_zikatanov_2017}.
\begin{lemma}\label{lemma::fictspace}
  Let $H, \HT$ be two real Hilbert spaces equipped with norms induced by
  $A: H\rightarrow H^*$ and $\AT: \HT\rightarrow \HT^*$
  and let there exist a linear operator $\Pi: \HT\rightarrow H $ such that the continuity condition
  \begin{align}\label{eq::fict_cont}
    \norm{\Pi \tilde{v}}_{A}^2 \leq c_0\norm{\tilde{v}}_{\AT}^2 \quad\quad\forall\tilde{v}\in\tilde{H}
  \end{align}
  and the stability condition,
  \begin{align}\label{eq::fict_stab}
    \forall v\in H~ \exists \tilde{v}\in\HT \quad\text{such that}\quad v=\Pi\tilde{v}
    \quad\text{and}\quad \norm{\tilde{v}}_{\AT}^2 \leq c_1 \norm{v}_A^2,
  \end{align}
  hold.
  Then, for the preconditioner $\hat{A}_{\textrm{a}}$ defined by
  $\hat{A}_{\textrm{a}}^{-1} := \Pi\AT^{-1}\Pi^*$,
  there holds the spectral estimate 
  \begin{align}\label{eq::fict_equiv}
  c_0^{-1}  (v, v)_{A} \leq (\hat{A}_{\textrm{a}}^{-1} A v, v)_A \leq c_1 (v, v)_{A} \quad\forall v\in H.
  \end{align}
\end{lemma}
The term auxiliary space method, as coined in \cite{XuAux_1996},
refers to the case where the titular fictitious space $\HT$ is a product space that contains $H$
itself as a component, $\HT = H \times V_1 \times\ldots\times V_n$,
so $\Pi$ takes the form $\Pi=(I | \Pi_1 | \ldots | \Pi_n)$,
and $\tilde{A} := \textrm{diag}(M,\bar{A}_1,\ldots,\bar{A}_n)$ is a diagonal operator with $M:H\rightarrow H^*$ and $\bar{A}_j:V_j\rightarrow V_j^*$ and induced norm $\|(v, v_1, \ldots, v_n)\|^2_{\tilde{A}} = \|v\|_M^2 + \sum_{j=1}^n \|v_j\|_{\bar{A}_j}^2$.
The stability condition \eqref{eq::fict_stab} then demands the existence of a
stable decomposition $v = v_0 + \sum_{j=1}^n v_j$ with $v_0\in H$
and $v_j$ in the range of $\Pi_j$.
The underlying idea is that the remainder $v_0\in H$ in this composition
is small and somehow localized and $M$ can be a computationally cheap method (or ``smoother'').
Often, $M$ is given by some form of additive or multiplicative Schwarz method
such as (Block-)Jacobi or (Block-)Gauss-Seidel.
In the only relevant case here, where all involved spaces are finite dimensional and $n=1$,
the ASP $\hat{A}_{a}$ is just
\begin{align*}
\hat{A}_{a}^{-1} =\Pi\AT^{-1}\Pi^* = \Pi
  \left(
  \begin{matrix}
    M^{-1} & 0 \\
    0 & \bar{A}_1^{-1}
  \end{matrix}
  \right)\Pi^*
  =
  M^{-1} + \Pi_1 \tilde{A}_1^{-1} \Pi_1^*.
\end{align*}
As alluded to by the subscript, $\hat{A}_{a}$ is an additive preconditioner in that,
given some right hand side vector $b$ and intermediate approximation $x^0$
with residual $r^0 := b - Ax^0$, one Richardson iteration with preconditioner 
$\hat{A}_{\textrm{a}}$ is to perform
\begin{align*}
  x^0\rightarrow x^0 +  M^{-1} r^0 + \Pi_1 \tilde{A}_1^{-1} \Pi_1^*r^0,
\end{align*}
i.e. to perform two updates additively.
The multiplicative ASP $\hat{A}_{\textrm{m}}$ is implicitly defined by
performing these updates successively instead, 
\begin{align*}
  x^1 := x^0 + M^{-1} r^0, &\quad\quad r^1 := b - A x^1, \\
  x^2 := x^1 + \Pi_1 \tilde{A}_1^{-1} \Pi_1^*r^1, &\quad\quad r^2 := b - A x^2,
\end{align*}
and then, performing another smoothing step with the adjoint smoother $M^*$
\begin{align*}
  x^3 := x^2 + (M^*)^{-1} r^2
\end{align*}
yielding $x^3:= x^0 + \hat{A}_{\textrm{m}}^{-1} r^0$.
Multiplication with $\hat{A}_{\textrm{m}}^{-1}$ is just performing this procedure once
starting with $x^0=0$.
With symmetry ensured by the additional smoothing step,
positive definiteness of $\hat{A}_{\textrm{m}}$ follows from $A \leq M$
and $\Pi_1 \tilde{A}_1^{-1} \Pi_1^* \leq A^{-1}$
which can always be achieved by scaling the component preconditioners.
If $M$ is a (Block-)Jacobi preconditioner, scaling
of $M$ can be avoided by replacing it with the corresponding
(Block-)Gauss-Seidel iteration which never over-corrects, see
\cite{Xu_subspace}.
\begin{lemma} \label{lemma::add_to_mult}
  Let an ASP $\hat{A}_{\textrm{a}}$ with $n=1$ fulfill the conditions of Lemma \ref{lemma::fictspace},
  and $M$ be either self-adjoint and positive definite with $M \leq A$ or
  given by (Block-)Gauss-Seidel iterations.
  Let $\tilde{A}_1$ self-adjoint and positive definite with 
  \begin{align}\label{eq::add_to_mult_corrections}
  \Pi_1^* A \Pi_1  \leq \tilde{A}_1.
  \end{align}
  Then $\hat{A}_{\textrm{m}}$ is self-adjoint and positive definite and there holds
  \begin{align}
    c_1^{-1}\hat{A}_{\textrm{m}} \lesssim A \leq \hat{A}_{\textrm{m}}. \label{eq::add_to_mult}
  \end{align}
\end{lemma}
\begin{proof} 
  Can be shown within the framework of space decomposition and subspace correction,
  see \cite{Xu_subspace}.
  The analysis there rests on a strengthened Cauchy-Schwarz type inequality and a stable decomposition.
  The former is implied by limited overlap of subspaces and the additional requirements posed on $M$ and $\tilde{A}_1$
  and the latter is directly related to \eqref{eq::fict_stab}.
  See also the discussion in \cite[Section 6]{xu_zikatanov_2017}, where convergence bounds for multiplicative
  two-grid Algebraic Multigrid methods are derived from the fictitious space lemma.
\end{proof}

\section{Preconditioners for $A$} \label{sec::aux}
From the point of view of Section \ref{ssec::aux_frame}, a straightforward
approach to preconditioning $\Ss$ is to use the conforming low order space
$\Vbar_h$,
where preconditioning is well understood and efficient and scalable
software is widely available, as basis for an ASP.
A slight complication in the analysis arises due to the non-conformity
in boundary conditions between $\VUUH = V_h\times\Vhat_h$,
where tangential Dirichlet conditions on $\GammaNT$ are imposed in $\hat{V}_h$,
and $\Vbar_h$, where $\GammaNT$ does not feature any Dirichlet conditions.
While imposing strong tangential Dirichlet conditions in $\Vbar_h$
would sidestep the issue and be convenient for theory,
in practice this is only a simple matter when the outflow lies
in an axis-aligned plane and we can impose Dirichlet
conditions in the $x,y$ or $z$ component.
Therefore, we for now assume that $\GammaNT=\emptyset$ and address
the case $\GammaNT\neq\emptyset$ separately in Lemma \ref{lemma::outflow}
at the end of this Section.

On $\Vbar_h$, we define the bilinear form $\bar{a}(\cdot,\cdot)$
(as usual, with associated operator $\bar{A}$ and
Finite Element matrix $\bar{\mA}$) by
\begin{align} \label{eq::defAhat}
\bar{a}(\ubh, \vbh) := \nu^{-1}\int_{\Omega}\eps(\ubh):\eps(\vbh)\dx\quad\forall\ubh,\vbh\in\Vbar_h.
\end{align}
To define the operator $\Pi$ in \eqref{eq::fict_cont}
we need the embedding operator
\begin{align} \label{eq::defE}
  E : \Vbar_h\rightarrow \VUUH : \bar{u}_h\mapsto (\bar{u}_h, (\bar{u}_h)_t)
\end{align}
with associated Finite Element matrix $\mE$.
\begin{corollary}\label{coro::embed}
  For $\ub_h\in\Vbar_h$ there holds
  \begin{align}
    \|E\ub_h\|_{\Ss} = \|\bar{u}_h\|_{\bar{A}}. \label{eq::A_eq_ETSE}
  \end{align}
\end{corollary}
\begin{proof}
  For $\bar{u}_h\in \Vbar_h$ and $\GammaNT=\emptyset$ there holds $E\ub_h\in \VUUHC$ from Lemma \ref{lemma::SPD}
  and \eqref{eq::A_eq_ETSE} follows from \eqref{eq::schur_SPD_cont}.
\end{proof}
To establish the stable decomposition \eqref{eq::fict_stab}
we use $\IV$ from Lemma~\ref{lemma::nc_to_c_nokorn_withbc} and define
\begin{align} \label{eq::defIVB}
  \IU : \VUUH\rightarrow \Vbar_h : (u_h, \hat{u}_h) \mapsto \IV u_h.
\end{align}
\begin{corollary}\label{coro::nc_to_c_approx}
  For $(u_h, \hat{u}_h)\in\VUUH$ and $(w_h,\hat{w}_h):=(I-E\IU)(u_h, \hat{u}_h)$ there holds
  \begin{align}\label{eq::nc_to_c_approx}
    \sum_{T\in\mesh}\Big(\|\eps(w_h)\|_T^2 + h^{-2}\|w_h\|_{T}^2
    + \sum_{F\in\facetsT}\| \Pi^{\pdg-1}(w_h - \hat{w}_h)_t\|_{j,F}^2\Big)
    \lesssim \|(u_h, \hat{u}_h)\|_{\eps, h}^2,
  \end{align}
\end{corollary}
\begin{proof}
  Per definition of $E$
  there holds $\Pi^{\pdg-1}(w_h - \hat{w}_h)_t = \Pi^{\pdg-1}(u_h - \hat{u}_h)_t$
  and the facet terms are bounded trivially.
  The volume terms are bounded with Lemma \ref{lemma::nc_to_c_nokorn_withbc}
  and the identity for the jump terms in \eqref{eq::jFjumpN}.
\end{proof}
With $\bar{V}_h$ being of low order, robustness in the polynomial degree $\pdg$
has to be achieved by the smoother.
\begin{theorem}\label{theorem::aux_pc}
  Let $\mM$ be the overlapping Block-Jacobi preconditioner for $\SsM$
  that has one block per facet $F\in\facets$ that contains all
  $\VUUH$ degrees of freedom associated to either $F$ or any
  $T\in\mesh$ such that $F\in\facetsT$.
  Let $\mC$ be an SPD preconditioner for $\Abm$ such that
  $\mC \sim \Abm$.
  Then, conditions \eqref{eq::fict_cont} and \eqref{eq::fict_stab} of Lemma \ref{lemma::fictspace}
  are fulfilled for $H = \VUUH$, $\HT=\VUUH\times\Vbar_h$,
  \begin{align*}
    \Pi:\HT\rightarrow H : ((u_h, \hat{u}_h), \bar{u}_h) \mapsto (u_h, \hat{u}_h) + E\bar{u}_h,
  \end{align*}
  and $\tilde{A}:=\operatorname{diag}(M, C)$ with $c_0\lesssim 1$ and $c_1\lesssim\gamma\cdot(\log{\pdg})^3$.
  That is, for $\SsMHA^{-1} := \mM^{-1} + \mE \mC^{-1} \mE^T$
  there holds
  \begin{align*}
    \SsMHA \lesssim  \SsM \lesssim \gamma\cdot(\log{\pdg})^{3} \SsMHA.
  \end{align*}
\end{theorem}
We postpone the proof of Theorem \ref{theorem::aux_pc} to Section
\ref{ssec::schur_schur}, where we discuss preconditioning of $\SsD$,
as obtaining the logarithmic bound in $\pdg$ is more natural
in that context.
\begin{remark}
  If one is satisfied with a polynomial bound in $\pdg$,
  Theorem \ref{theorem::aux_pc} can be shown only using
  standard Finite Element inverse estimates
  and Corollaries \ref{coro::embed} and \ref{coro::nc_to_c_approx}.
\end{remark}

\subsection{Preconditioning via the condense system}\label{ssec::schur_schur}

Using the factorization in \eqref{eq::factorSsigma} to implement multiplication with $\SsM$
also opens up a way to precondition $\SsM$,
where replacing $\SsDM$ by some preconditioner $\SsDMH$
yields a preconditioner
\begin{align*}
  \SsMH^{\text{ext}} :=
  \begin{pmatrix}
    \mI&\SsM_{\partial\circ}\SsM_{\circ\circ}^{-1}\\
    \mZ&\mI\\
  \end{pmatrix}
  \begin{pmatrix}
    \SsDMH & \mZ \\
    \mZ&\SsM_{\circ\circ}\\
  \end{pmatrix}
  \begin{pmatrix}
    \mI&\mZ\\
    \SsM_{\circ\circ}^{-1}\SsM_{\circ\partial}&\mI\\
  \end{pmatrix}
\end{align*}
for $\SsM$.
From the factorization \eqref{eq::factorSsigma} it clearly follows that
\begin{align*}
  c_1 \SsMH^{\text{ext}} \leq \SsM \leq c_2 \SsMH^{\text{ext}}
  \quad\Leftrightarrow\quad
  c_1 \SsDMH \leq \SsDM \leq c_2 \SsDMH
\end{align*}
and we are left with the task to precondition the ``double'' Schur
complement $\SsDM$. 
Analogues for $\SsDM$ of the preconditioners $\SsMHA$ and $\SsMHM$ can
be constructed straightforwardly with the modified embedding operator
\begin{align} \label{eq::defED}
  E^{\partial} : \Vbar_h\rightarrow \VUUHH : \bar{u}_h\mapsto \mathcal{H}E\bar{u}_h.
\end{align}
Note that the matrix $\mDE\GalEqD E^{\partial}$ is just a sub-matrix of $\mE$
as $\mathcal{H}$ does not change $(\vu^{\partial},\vuh)$ coefficients, see \eqref{eq::Hcoeffs},
that is simply
\begin{align*}
  \mE =
  \begin{pmatrix}
    \mE_{\circ} \\
    \mE_{\partial} \\
    \mE_{\vuh} \\
  \end{pmatrix}
  \quad\quad
  \mE^{\partial} =
  \begin{pmatrix}
    \mE_{\partial} \\
    \mE_{\vuh} \\
  \end{pmatrix}.
\end{align*}
We \emph{could} modify the bilinear form in $\Vbar_h$ and use
$\Abm:=\mE^{\partial,T}\SsDM\mE^{\partial}$, which would be computable element-wise.
In that case, the exact analogue of Corollary \ref{coro::embed} would hold.
However, as we now show, this is not strictly necessary and
for ease of implementation we opt to keep $\Abm$ defined by \eqref{eq::defAhat}.
\begin{lemma}\label{lemma::I_bubble}
  For $(u_h,\hat{u}_h)\in\VUUH$ there holds
  \begin{align}
    \IV u^{\circ}_h &= 0 , \label{eq::I_bubble} \\
    \IU(u_h,\hat{u}_h) &= \IU\mathcal{H}(u_h,\hat{u}_h) = \IU\mathcal{H}_{\eps}(u_h,\hat{u}_h). \label{eq::IHeqI}
  \end{align}
\end{lemma}
\begin{proof}
  Any $u^{\circ}_h\in V_h^{\circ}$ restricted to $T\in\mesh$ is a
  normal bubble. At any vertex $p$ of $T$, $\spacedim$ linearly
  independent components of $(u_h^{\circ})_{|T}(p)$ vanish, and
  therefore $(u_h^{\circ})_{|T}(p)$ and $\IV u_h^{\circ}$ also vanish as a
  whole.
  This concludes the proof as $\mathcal{H},\mathcal{H}_{\eps}$ only add
  some $v_h\in V_h^{\circ}$ to the $V_h$ component of $(u_h,\hat{u}_h)\in\VUUH$.
\end{proof}

\begin{corollary}\label{coro::embed_schur}
  For $\bar{u}_h\in \Vbar_h$ there holds
  \begin{align}\label{eq::embed_schur}
    \gamma^{-1}\|\bar{u}_h\|_{\bar{A}}
    \lesssim \gamma^{-1} \nu\|\mathcal{H}_{\eps}E\bar{u}_h\|_{\eps,h,\partial}
    \lesssim \|E^{\partial}\bar{u}_h\|_{\Ss^{\partial}}
    \leq \|\bar{u}_h\|_{\bar{A}}.
  \end{align}
\end{corollary}
\begin{proof}
  The sharp upper bound is a
  consequence of the energy minimization
  \eqref{eq::doubleschurinf} and Corollary \ref{coro::embed},
  \begin{align*}
    \|E^{\partial}\bar{u}_h\|_{\SsD}^2
    = \|\mathcal{H}E\bar{u}_h\|_{\Ss}^2
    \leq \|E\bar{u}_h\|_{\Ss}^2
    =\|\bar{u}_h\|_{\bar{A}}^2.
  \end{align*}
  With $ (u_h,\hat{u}_h) := \mathcal{H}_{\eps}E\bar{u}_h$ and \eqref{eq::I_bubble}
  we conclude the identity
  $\IV u_h = \IV\mathcal{H}_{\eps} E \bar{u}_h = \IV E \bar{u}_h = \bar{u}_h$.
  Now, \eqref{eq::nc_to_c_nokorn_withbc} and \eqref{eq::jFjumpN} show
  \begin{align*}
    \|\bar{u}_h\|_{\bar{A}}^2 = 
    \nu\|\eps(\IV u_h)\|_0^2
    &\lesssim \nu\|\eps(u_h)\|_0^2  + \nu\|\eps(\IV u_h - u_h)\|_0^2 
     \lesssim \nu\|(u_h,\hat{u}_h)\|_{\eps,h}^2
    = \nu\|\mathcal{H}_{\eps}E\bar{u}_h\|_{\eps,h} ^2
  \end{align*}
  and the rest follows form the lower bound in \eqref{eq::schur_schur_SPD}.
\end{proof}
\begin{corollary}\label{coro::nc_to_c_approx_schur}
  For $(u_h, \hat{u}_h)\in\VUUH$ and either
  $(w_h, \hat{w}_h) := \mathcal{H}_{\eps}(I-E\IU)(u_h, \hat{u}_h)$
  or $(w_h, \hat{w}_h) := \mathcal{H}(I-E\IU)(u_h, \hat{u}_h)$
  there holds
  \begin{align}\label{eq::nc_to_c_approx_schur}
    \sum_{T\in\mesh}\Big(\|\nabla w_h\|_T^2 + h^{-2}\|w_h\|_{T}^2
    + \sum_{F\in\facetsT}\| \Pi^{\pdg-1}(w_h - \hat{w}_h)_t  \|_{j,F}^2\Big)
    \lesssim \|(u_h, \hat{u}_h)\|_{\eps, h, \partial}^2.
  \end{align}
\end{corollary}
\begin{proof}
  With the readily apparent $\IU E\IU= \IU$ and \eqref{eq::IHeqI} we see
  \begin{align*}
    E\IU\mathcal{H}_{\eps}(I-E\IU) = 0
    \quad\text{and}\quad
    \mathcal{H}_{\eps}(I-E\IU)(u_h, \hat{u}_h) = \mathcal{H}_{\eps}(I-E\IU)\mathcal{H}_{\eps}(u_h, \hat{u}_h).
  \end{align*}
  This lets us insert a zero into $(w_h, \hat{w}_h)$ to obtain an expression without $\mathcal{H}_{\eps}$ in front,
  \begin{align*}
    (w_h, \hat{w}_h)
    = \mathcal{H}_{\eps}(I-E\IU)(u_h,\hat{u}_h)
    = (I-E\IU)\mathcal{H}_{\eps}(I-E\IU)\mathcal{H}_{\eps}(u_h,\hat{u}_h).
  \end{align*}
  Corollary \ref{coro::nc_to_c_approx}
  applied to $\mathcal{H}_{\eps}(I-E^{\partial}\IU)\mathcal{H}_{\eps}(u_h,\hat{u}_h)$ shows
  \begin{align*} 
    \sum_{T\in\mesh} \Big(h^{-2}\|w_h\|_0^2 +& \|\nabla w_h\|_0^2
    + \sum_{F\in\facetsT}h^{-1}\| \Pi^{\pdg-1}(w_h - \hat{w}_h)_t  \|_{j,F}^2\big) 
    \lesssim  \|\mathcal{H}_{\eps}(I-E\IU)\mathcal{H}_{\eps}(u_h,\hat{u}_h)\|_{\eps,h}^2.
  \end{align*}
  The proof is concluded by the energy minimization of $\mathcal{H}_{\eps}$,
  \begin{align*}
    \|\mathcal{H}_{\eps}(I-E\IU)\mathcal{H}_{\eps}(u_h,\hat{u}_h)\|_{\eps,h}^2
    & \leq \|(I-E\IU)\mathcal{H}_{\eps}(u_h,\hat{u}_h)\|_{\eps,h}^2 
     \lesssim \|\mathcal{H}_{\eps}(u_h,\hat{u}_h)\|_{\eps,h}^2 
     = \|(u_h,\hat{u}_h)\|_{\eps,h,\partial}^2,
  \end{align*}
  where the continuity of $E\IU$ in the $\|\cdot\|_{\eps,h}$ norm follows from
  Lemma \ref{lemma::nc_to_c_nokorn_withbc}.
  The other case $(w_h, \hat{w}_h) := \mathcal{H}(I-E\IU)(u_h, \hat{u}_h)$ works analogously.
\end{proof}

An operator that, like $E\IU$, extracts a low order component out of
$(u_h,\hat{u}_h)\in\VUUH$ is
\begin{align}\label{def::projlo}
    \projlo:=\begin{cases}
      \VUUH \rightarrow \big(V_h\cap\Poly^1(\mesh, \rr^d)\big)\times\big(\hat{V}_h\cap\Poly^1(\facets,\rr^{\spacedim})\big)\\
      (u_h, \hat{u}_h)\mapsto (\mathcal{I}_{\BDM}^1u_h, \Pi^1_F \hat{u}_h),
      \end{cases}
\end{align}
where $\mathcal{I}_{\BDM}^1$ is the standard $\BDM^1$ interpolator, see \cite{brezzi2012mixed},
that is, $\forall F\in\facets$ there holds $(\mathcal{I}_{\BDM}^1u_h)_n = (\Pi^1_F(u_h))_n$.

\begin{corollary}\label{coro::projlo_approx}
  For $(u_h, \hat{u}_h)\in\VUUHH$ and
  $(w_h, \hat{w}_h) := \mathcal{H}_{\eps}(I-\projlo)(u_h, \hat{u}_h)$ there holds
  \begin{align}\label{eq::projlo_approx}
    \sum_{T\in\mesh}\Big(\|\nabla w_h\|_T^2 + h^{-2}\|w_h\|_{T}^2
    + \sum_{F\in\facetsT}\| \Pi^{\pdg-1}(w_h - \hat{w}_h)_t  \|_{j,F}^2\Big)
    \lesssim \|(u_h, \hat{u}_h)\|_{\eps, h, \partial}^2.
  \end{align}
\end{corollary}
\begin{proof}
  Follows from the Bramble-Hilbert Lemma, an element-level Korn inequality
  and Lemma \ref{lemma::I_bubble} with similar arguments as the
  previous Corollary \ref{coro::nc_to_c_approx_schur}.
\end{proof}

\begin{theorem}\label{theorem::aux_pc_schur}
  Let $\mDM$ be the block Jacobi preconditioner for $\SsDM$,
  consisting of one block per facet $F\in\facets$ that contains all
  $\VUUHH$ degrees of freedom associated to $F$.
  Let $\mC$ be an SPD preconditioner for $\Abm$ such that
  $\mC \sim \Abm$.
  Then, conditions \eqref{eq::fict_cont} and \eqref{eq::fict_stab} of Lemma \ref{lemma::fictspace}
  are fulfilled for $H = \VUUHH$, $\HT=\VUUHH\times\Vbar_h$,
  \begin{align*}
    \Pi:\HT\rightarrow H : ((u_h, \hat{u}_h), \bar{u}_h) \mapsto (u_h, \hat{u}_h) + E^{\partial}\bar{u}_h,
  \end{align*}
  and $\tilde{A}:=\operatorname{diag}(M^{\partial}, C)$ with $c_0\lesssim 1$ and $c_1\lesssim\gamma\cdot(\log{\pdg})^3$.
  That is, for $(\SsDMHA)^{-1} := (\mDM)^{-1} + \mDE \mC^{-1} \mDET$
  there holds
  \begin{align*}
    \SsDMHA \lesssim  \SsDM \lesssim \gamma\cdot(\log{\pdg})^{3} \SsDMHA.
  \end{align*}
\end{theorem}
\begin{proof}
  The continuity condition \eqref{eq::fict_cont} holds as
  $\|E^{\partial}\bar{u}_h\|_{\SsD} \leq \|\bar{u}_h\|_{\bar{A}}$
  is shown in Corollary \ref{coro::embed_schur}
  and $\|(u_h, \hat{u}_h)\|_{\SsD}^2 \lesssim \|(u_h, \hat{u}_h)\|_{M^{\partial}}^2$
  follows from limited overlap of basis functions.
  For some $(u_h, \hat{u}_h)\in H = \VUUH$, the choice  
  \begin{align*}
    \tilde{v}:= ((I - E^{\partial}\IU)(u_h, \hat{u}_h), \IU(u_h, \hat{u}_h)) \in \HT,
  \end{align*}
  fulfills $(u_h, \hat{u}_h) = \Pi \tilde{v}$.
  The stability condition \eqref{eq::fict_stab} is verified by showing
  \begin{align} \label{eq::aux_pc_schur::toshow}
    \|(I - E^{\partial}\IU)(u_h, \hat{u}_h)\|_{M^{\partial}}^2 + \|\IU(u_h, \hat{u}_h)\|_C^2
    \lesssim \gamma\cdot(\log{\pdg})^3 \|(u_h, \hat{u}_h)\|_{\SsD}^2.
  \end{align}
  For the second term, $C\lesssim \bar{A}$ and Corollary \ref{coro::embed_schur}
  bound it by $\|E^{\partial}\IU(u_h, \hat{u}_h)\|_{\eps,h,\partial}^2$ which is then further bounded by the $\|\cdot\|_{\SsD}$ norm
  with continuity of $E^{\partial}\IU$ in $\|\cdot\|_{\eps,h,\partial}$, as implied by Corollary \ref{coro::nc_to_c_approx_schur},
  and \eqref{eq::schur_schur_SPD} where we incur the factor $\gamma$.
  The other bound requires a more careful approach.
  For general $(v_h, \hat{v}_h)\in\VUUHH$, and therefore also for $(I - E^{\partial}\IU)(u_h, \hat{u}_h)\in\VUUHH$,
  the lower bound in \eqref{eq::schur_schur_SPD} shows
  \begin{align}\label{eq::aux_pc_schur_upper}
    \nu\sum_{F\in\facets} \|(v_h, \hat{v}_h)\|_{\eps,F}^2
    \lesssim
    \nu\|(v_h, \hat{v}_h)\|_{\eps,h,\partial}^2
    \lesssim
    \gamma \|(v_h, \hat{v}_h)\|_{\SsD}^2.
  \end{align}
  In the first step we bounded the facet terms in the sum, where an infimum is taken
  over functions with arbitrary traces on neighboring faces, 
  by $\|(v_h, \hat{v}_h)\|_{\eps,h,\partial}$, where these traces are fixed.
  On the other hand, the upper
  inequality in \eqref{eq::schur_schur_SPD} shows
  \begin{align}\label{eq::aux_pc_schur_lower}
    \|(v_h, \hat{v}_h)\|_{M^{\partial}}^2
    = \sum_{F\in\facets}\|(v_h, \hat{v}_h)^{(F)}\|_{\SsD}^2
    \lesssim \nu\sum_{F\in\facets}\|(v_h, \hat{v}_h)\|_{\eps,F,0}^2,
  \end{align}
  where $(v_h, \hat{v}_h)^{(F)}$ denotes
  the element of $\VUUHH$ that has the 
  same coordinates
  as $(v_h, \hat{v}_h)$ for degrees of freedom associated to $F$
  and whose degrees of freedom are zero otherwise
  (Galerkin isomorphism $\GalEqD$).
  Given the continuity of $E^{\partial}\IU$ in $\|\cdot\|_{\eps,h,\partial}$
  (see Corollary \ref{coro::nc_to_c_approx_schur}),
  the crucial step is therefore to bound $\|\cdot\|_{\eps,F,0}$ by
  $\|\cdot\|_{\eps,F}$, as in Corollary \ref{coro::traceZ}.
  However, Corollary \ref{coro::traceZ} is only applicable if $\Pi^R_F((v_h)_nn+\hat{v}_h) = 0~\forall F\in\facets$,
  which is not usually true for $(I - E^{\partial}\IU)(u_h, \hat{u}_h)$.

  This does not pose a problem for low-order functions
  or, crucially, their harmonic extensions,
  where an alternative path via an inverse inequality
  bypasses the trace estimate.
  For a low order $(v_h,\hat{v}_h)\in\VUUH \cap(\Poly^1(\mesh)\times\Poly^1(\facets))$,
  a standard, and necessarily $k$-independent, inverse estimate is
  \begin{align}\label{eq::aux_pc_schur::invest}
    \sum_{F\in\facets}\|(v_h,\hat{v}_h)^{(F)}\|_{\eps,h}^2
    \lesssim\sum_{T\in\mesh} h^{-2}\|v_h\|_{T}^2
    + \sum_{F\in\facets}\|\Pi^{\pdg-1}(v_h-\hat{v}_h)_t\|_{j,F}^2.
  \end{align}
  Because of the energy minimization in
  $\|\cdot\|_{\eps,h,\partial}$, the estimate holds with the same constant also
  for discrete harmonix extensions
  $(v_h,\hat{v}_h)\in\mathcal{H}_{\eps}\big(\VUUH \cap (\Poly^1(\mesh)\times\Poly^1(\facets))\big)$
  of these low order functions where
  $\|(v_h,\hat{v}_h)^{(F)}\|_{\eps,h} = \|(v_h,\hat{v}_h)^{(F)}\|_{\eps,h,\partial}$.
  For approximation errors, the right hand side can then further be bounded 
  Corollary~\ref{coro::nc_to_c_approx_schur} and Corollary~\ref{coro::projlo_approx}.

  Therefore, the strategy is to use the operator $\projlo$ as defined in \eqref{def::projlo}
  to split the $\|\cdot\|_{M^{\partial}}$ term in \eqref{eq::aux_pc_schur::toshow} into
  low and high order components.
  The former can then be bounded via the inverse estimate
  and the latter via the trace inequality, we have
  \begin{align*}
    \|(I - E^{\partial}\IU)(u_h, \hat{u}_h)\|_{M^{\partial}}^2 ~\lesssim~ 
    &\|\mathcal{H}\projlo(I - E^{\partial}\IU)(u_h, \hat{u}_h)\|_{M^{\partial}}^2 
    +  \|(I - \mathcal{H}\projlo)(I - E^{\partial}\IU)(u_h, \hat{u}_h)\|_{M^{\partial}}^2.
  \end{align*}
  As for $(v_h,\hat{v}_h)\in\VUUH$,
  the (low order) normal trace of the $V_h$ component of $E\IU(v_h, \hat{v}_h)$
  and the entire $\hat{V}_h$ component are not changed by $\mathcal{H}$ 
  there holds
  \begin{align*}
  \projlo E^{\partial}\IU(v_h, \hat{v}_h)
  = \projlo \mathcal{H}E\IU(v_h, \hat{v}_h)
  = \projlo E\IU(v_h, \hat{v}_h)
  = E\IU(v_h, \hat{v}_h).
  \end{align*}
  That is
  $\projlo E^{\partial}\IU = E\IU$, therefore $\mathcal{H}\projlo E^{\partial}\IU = \mathcal{H}E\IU = E^{\partial}\IU$
  and the high order term can be simplified, 
  \begin{align*}
    \|(I - \mathcal{H}\projlo)(I - E^{\partial}\IU)(u_h, \hat{u}_h)\|_{M^{\partial}}^2
    = & \|(I - \mathcal{H}\projlo)(u_h, \hat{u}_h)\|_{M^{\partial}}^2.
  \end{align*}
  Note that we can apply Corollary \ref{coro::traceZ} not only to $(I -
  \projlo)(u_h,\hat{u}_h)$, which is apparent from the definition of
  $\projlo$, but also to $(I - \mathcal{H}\projlo)(u_h,\hat{u}_h)$
  because again, as argued above, $\mathcal{H}$ does not change the relevant traces.
  Therefore, \eqref{eq::aux_pc_schur_lower}, \eqref{eq::traceZ}
  and then \eqref{eq::aux_pc_schur_upper} show
  \begin{align*}
    \|(I - \mathcal{H}\projlo)(u_h, \hat{u}_h)\|_{M^{\partial}}^2 
    \lesssim & \nu (\log{\pdg})^3 \|(I - \mathcal{H}\projlo)(u_h, \hat{u}_h)\|_{\eps,h,\partial}^2 
    \lesssim  \nu (\log{\pdg})^3 \|(u_h, \hat{u}_h)\|_{\eps,h,\partial}^2,
  \end{align*}
  where the continuity of $\mathcal{H}\projlo$ used in the last estimate follows from
  the Bramble Hilbert Lemma as in the proof of Corollary~\ref{coro::projlo_approx}.
  Finally, the bound
  \begin{align*}
    \|(I - \mathcal{H}\projlo)(I - E^{\partial}\IU)(u_h, \hat{u}_h)\|_{M^{\partial}}^2 \lesssim & \gamma\cdot(\log{\pdg})^3 \|(u_h, \hat{u}_h)\|_{\SsD}^2.
  \end{align*}
  follows with \eqref{eq::schur_schur_SPD}.
  As for the low order term,
  with
  $\mathcal{H}\projlo E^{\partial}\IU =
  \mathcal{H}E\IU$
  we see
  \begin{align*}
    \|\mathcal{H}\projlo(I - E^{\partial}\IU)(u_h, \hat{u}_h)\|_{M^{\partial}}^2
    & = \|\mathcal{H}(\projlo - E\IU)(u_h, \hat{u}_h)\|_{M^{\partial}}^2
  \end{align*}
  and applying \eqref{eq::aux_pc_schur::invest} to the (harmonic extension of)
  the low order function $(\projlo - E\IU)(u_h, \hat{u}_h)$ gives
  \begin{align*}
    \|\mathcal{H}\projlo(I - E^{\partial}\IU)(u_h, \hat{u}_h)\|_{M^{\partial}}^2
    \leq \nu &\sum_{F\in\facets} \|\big(\mathcal{H}(\projlo - E\IU)(u_h, \hat{u}_h)\big)^{(F)}\|_{\eps,h,\partial}^2 \\
    = \nu&\sum_{F\in\facets} \|\big(\mathcal{H}_{\eps}(\projlo - E\IU)(u_h, \hat{u}_h)\big)^{(F)}\|_{\eps,h}^2 \\
    \lesssim \nu&  \Big(\sum_{T\in\mesh} h^{-2}\|w_h\|_T^2 
         + \!\! \sum_{F\in\facetsT} \|\Pi^{\pdg-1}(w_h - \hat{w}_h)_t\|_{j,F}^2\Big),
  \end{align*}
  where we write $(w_h,\hat{w}_h):=\mathcal{H}_{\eps}(\projlo - E\IU)(u_h, \hat{u}_h)$.
  We further split $(w_h,\hat{w}_h)$ into
  $(\alpha_h,\hat{\alpha}_h) := \mathcal{H}_{\eps}(I - E\IU)(u_h, \hat{u}_h)$
  and
  $(\beta_h,\hat{\beta}_h) := \mathcal{H}_{\eps}(I - \projlo)(u_h, \hat{u}_h)$
  and get
  \begin{align*}
  \|\mathcal{H}\projlo(I -
    E^{\partial}\IU)(u_h, \hat{u}_h)\|_{M^{\partial}}^2
    \lesssim & \nu \Big(\sum_{T\in\mesh} h^{-2}\|\alpha_h\|_T^2
               + \!\! \sum_{F\in\facetsT} \|\Pi^{\pdg-1}(\alpha_h - \hat{\alpha}_h)_t\|_{j,F}^2 \Big)\\
             & + \nu \Big(\sum_{T\in\mesh} h^{-2}\|\beta_h\|_T^2
               +\!\!\sum_{F\in\facetsT} \|\Pi^{\pdg-1}(\beta_h - \hat{\beta}_h)_t\|_{j,F}^2 \Big).
  \end{align*}
  Corollary \ref{coro::nc_to_c_approx_schur} and
  \eqref{eq::schur_schur_SPD} bound the former two terms,
  \begin{align*}
    \nu \Big(\sum_{T\in\mesh} h^{-2}\|\alpha_h\|_T^2
    + \!\! \sum_{F\in\facetsT} \|\Pi^{\pdg-1}(\alpha_h - \hat{\alpha}_h)_t\|_{j,F}^2\Big)
    \lesssim \nu\|(u_h, \hat{u}_h)\|_{\eps,h,\partial}
    \lesssim \gamma \|(u_h, \hat{u}_h)\|_{\SsD},
  \end{align*}
  and Corollary~\ref{coro::projlo_approx} and
  \eqref{eq::schur_schur_SPD} the latter two,
  \begin{align*}
    \nu \Big( h^{-2}\|\beta_h\|_T^2
    + \sum_{F\in\facetsT} \|\Pi^{\pdg-1}(\beta_h - \hat{\beta}_h)_t\|_{j,F}^2 \Big)
    \lesssim \nu \|(u_h, \hat{u}_h)\|_{\eps,h,\partial}
    \lesssim \gamma \|(u_h, \hat{u}_h)\|_{\SsD}.
  \end{align*}
\end{proof}

\begin{proof}[Proof of Theorem \ref{theorem::aux_pc}]
  Similarly to the proof of Theorem \ref{theorem::aux_pc_schur},
  the continuity condition follows from Corollary \ref{coro::embed},
  limited overlap of basis functions and this time also limited
  overlap of the Jacobi blocks themselves.
  Also similarly, the stability condition is proven by setting
  $\tilde{v}:= ((I - E\IU)(u_h, \hat{u}_h), \IU(u_h, \hat{u}_h)) \in \HT$
  and using Corollary \ref{coro::nc_to_c_approx}.
  The bound $\|(I - E\IU)(u_h, \hat{u}_h)\|_{M}^2 \lesssim \|(u_h,\hat{u}_h)\|_{A}^2$
  follows from 
  \begin{align*}
    \|(I - E^{\partial}\IU)(u_h, \hat{u}_h)\|_{M}^2 \lesssim  \gamma\cdot(\log{\pdg})^3\|(u_h,\hat{u}_h)\|_{A}^2,
  \end{align*}
  which was already shown in the proof of Theorem \ref{theorem::aux_pc_schur}, and the estimate
  \begin{align*}
    \|(E - E^{\partial})\IU(u_h, \hat{u}_h)\|_{M}^2
    = \|(I-\mathcal{H})E\IU(u_h, \hat{u}_h)\|_{M}^2  \lesssim \|(u_h,\hat{u}_h)\|_{A}^2.
  \end{align*}
  It holds because $(I-\mathcal{H})E\IU(u_h, \hat{u}_h)$
  is a normal bubble, that is all its coupling degrees of freedom are zero,
  and $\Ss$ restricted to such functions is block diagonal.
\end{proof}
\begin{corollary}\label{coro::aux_pc_mult}
  Let $\SsMHM$ and $\SsDMHM$ be the multiplicative versions of $\SsMH$ and $\SsDMH$, respectively,
  with the Block-Jacobi smoothers $\mM$, $\mDM$ replaced by Block-Gauss-Seidel sweeps
  and let $\Abm \leq \mC$.
  Then there holds
  \begin{align}
    \gamma^{-1}\cdot(\log{\pdg})^{-3}\SsMHM  \lesssim &\SsM \leq \SsMHM, \label{eq::aux_pc_mult} \\
    \gamma^{-1}\cdot(\log{\pdg})^{-3}\SsDMHM \lesssim &\SsDM \leq \SsDMHM. \label{eq::aux_pc_mult_schur}
  \end{align}
\end{corollary}
\begin{proof}
  The former result \eqref{eq::aux_pc_mult} follows from Theorem \ref{theorem::aux_pc}
  and Lemma \ref{lemma::add_to_mult}, where
  condition \eqref{eq::add_to_mult_corrections} is fulfilled
  due to $\Abm\leq\mC$ and Corollary \ref{coro::embed}.
  The latter one \eqref{eq::aux_pc_mult_schur} follows along the same lines
  with Theorem \ref{theorem::aux_pc_schur}
  and the strict upper bound in \eqref{eq::embed_schur} for Lemma \ref{lemma::add_to_mult}.
\end{proof}

\begin{remark}
  Although we have only experimental evidence that the constant $\gamma$ in
  Theorem \ref{theorem::aux_pc} and Theorem \ref{theorem::aux_pc_schur}
  is benign, the proofs of these theorems show that
  in the $\|\cdot\|_{\eps,h}$ and $\|\cdot\|_{\eps,h,\partial}$ norm
  they hold independently of $\gamma$.
  That is, we have results for ASPs for HDG methods with optimal stabilization
  that are explicit and robust in $\pdg$.
\end{remark}



\subsection{Non-conformity in Boundary Conditions}

We now return to the case of $\GammaNT\neq\emptyset$.
Instead of enforcing zero tangential Dirichlet conditions on $\GammaNT$
in $\bar{V}_h$, it suffices to add a tangential penalty to $\bar{A}$
and for $E$ to zero out $\hat{V}_h$ degrees of freedom on $\GammaNT$.

\begin{lemma}\label{lemma::outflow}
  For some $C>0$, let $\bar{A}$ be defined by the modified bilinear form
  \begin{align*}
    \bar{a}(\bar{u}_h, \bar{v}_h) :=
    \int_{\Omega}\nu\eps(\bar{u}_h):\eps(\bar{v}_h)\dx + \sum_{F\subseteq\GammaNT} \int_F  \frac{\nu C\pdg^2}{h} (\bar{u}_h)_t(\bar{v}_h)_t\ds.
  \end{align*}
  and $\pi_0:\VUUH\rightarrow\VUUH$ be the operator that zeros out
  $\hat{V}_h$ degrees of freedom on $\GammaNT$.
  Then, for $C$ large enough there holds
  \begin{align*}
    \|\bar{u}_h\|_{\bar{A}}^2 &\lesssim \|\pi_0E\bar{u}_h\|_{\Ss}^2 \leq \|\bar{u}_h\|_{\bar{A}}^2, \quad \textrm{and} \quad
    \|\bar{u}_h\|_{\bar{A}}^2 \lesssim \|\mathcal{H}\pi_0E\bar{u}_h\|_{\SsD}^2 \leq \|\bar{u}_h\|_{\bar{A}}^2.
  \end{align*}
  These estimates are robust in $\pdg$.
\end{lemma}
\begin{proof}
  With the upper bound in \eqref{eq::schur_SPD} and \eqref{eq::jFjumpN} there holds
  \begin{align*}
    \| \pi_0E\bar{u}_h \|_{\Ss}^2
    &\leq \nu\Big(\|\eps(\bar{u}_h)\|_0^2 + \sum_{F\in\facetsNT} \| \Pi^{\pdg-1}(\bar{u}_h)_t\|_{j,F}^2 \Big)
    \lesssim \nu \Big(\|\eps(\bar{u}_h)\|_0^2 + \sum_{F\in\facetsNT} \frac{\pdg^2}{h} \| \Pi^{\pdg-1}(\bar{u}_h)_t\|_F^2 \Big),
  \end{align*}
  that is for large enough $C$ we have $\|\pi_0E\bar{u}_h\|_{\Ss}^2 \leq \|\bar{u}_h\|_{\bar{A}}^2$.
  The lower bound $ \|\bar{u}_h\|_{\bar{A}}^2 \lesssim \|\pi_0E\bar{u}_h\|_{\Ss}^2$
  similarly follows from the lower bound in \eqref{eq::schur_SPD} and the fact that,
  as $\bar{u}_h\in\Poly^1(\mesh)$ is of low order,
  the high order terms in \eqref{eq::jFjumpN} vanish and there holds
  \begin{align*}
     \frac{\pdg^2}{h} \| \Pi^{\pdg-1}(\bar{u}_h)_t\|_F^2 \sim \| \Pi^{\pdg-1}(\bar{u}_h)_t\|_{j,F}^2
  \end{align*}
  with a $\pdg$-robust constant.
  The estimates for the $\|\cdot\|_{\SsD}$-norm follow from the ones for the $\|\cdot\|_{\Ss}$-norm
  with energy minimization as in the proof of Corollary \ref{coro::embed_schur}.
\end{proof}

Modifying $\bar{A}$ and the embedding operators like this on shows the
proofs of Section \ref{sec::aux} also for the case
$\GammaNT\neq\emptyset$.

\section{The lowest order case} \label{sec::lorder}
The MCS method of Section \ref{sec::mcsstokes} is, as already mentioned there,
not stable in the lowest order case $\pdg=1$.
Stability of the method is recovered when a simplified stress tensor $\sigma = -\nu\nabla(u)$
is used in \eqref{eq::stokes_A}, but we are interested in treating the full
symmetric stress tensor $\sigma = -\nu\eps(u)$.
For that, the five coupling degrees of freedom per facet we have with $\pdg=1$,
three in $V_h\subseteq\BDM^1$ and two enforced by $\hat{V}_h$,
are too few to capture the six rigid body modes.

In \cite{gopalakrishnan2021minimal}, this was remedied by using a vector-valued $W_h$ instead of the $\kk$-valued
one here, which just means that all occurrences of $\omega_h$ have to be replaced by $\kappa(\omega_h)$ everywhere,
and taking it as a subset of $H(\div)$,
\begin{align*} 
  W_h &:= \{ \omega_h \in H_{0,D}(\div, \Omega): (\omega_h)_{|T} \in \Poly^0(T, \rr^3) + x  \Poly^0(T,\rr)~\forall T \in \mesh \},
\end{align*}
providing the missing coupling degree of freedom per facet.
Motivated by the fact that the divergence of $\omega_h = \curl(u)\in H(\div,\Omega)$ vanishes for the true solution $u\in H^1(\Omega,\rr^d)$, 
a consistent stabilizing term $(\div(\omega_h),\div(\eta_h))_0$ was added to the bilinear form.
We only briefly sketch how to adapt the
preconditioners and their analysis developed here.
Since $W_h\subseteq H(\div,\Omega)$ has a coupling degree of freedom per facet,
$\omega_h$ remains after static condensation and
$\Ss$ is a system for $(u_h,\hat{u}_h,\omega_h) \in \LOVUUH := V_h\times\hat{V}_h\times W_h$.
The norm in $\LOVUUH$ is
\begin{align*}
  \|(u_h,\hat{u}_h,\omega_h)\|_{\eps,h,\text{lo}}^2 :=
  \sum_{T\in\mesh}\Big( \|\eps(u_h)\|_T^2 + \sum_{F\in\facetsT} &h^{-1}\|\Pi_F^0(u_h - \hat{u})_h\|_{F}^2 
  + h\|(\curl(u_h) - \omega_h)_n\|_{F}^2 \Big),
\end{align*}
this is justified by the discrete Korn inequality
\begin{align*}
  \sum_{T\in\mesh} \| \nabla u_h \|_T^2 \lesssim
  \sum_{T\in\mesh} \|\eps(u_h)\|_T^2 + 
       \sum_{F\in\facets}h^{-1}\big\|\Pi^0_F\jump{u_h}_t\big\|_F^2
      + h\big\|\jump{n \cdot \curl(u_h)}\big\|_F^2
\end{align*}
introduced in \cite[Lemma 3.1]{gopalakrishnan2021minimal}.
We only need to adapt the "embedding" operator $E$ which now has a $W_h$
component and projects into the $\hat{V}_h$
component as for $u_h\in\Vbar_h$ the piecewise $\Poly^1$ tangential trace $(\bar{u}_h)_t\notin \hat{V}_h$,
\begin{align*} 
  E : \Vbar_h\rightarrow \LOVUUH : \bar{u}_h\mapsto (\bar{u}_h, \Pi_F^0(\bar{u}_h)_t, \curl(\bar{u}_h)).
\end{align*}
The analysis also needs to be only slightly modified using the equivalence
\begin{align*}
  \sum_{T\in\mesh} \|\eps(u_h)\|_T^2
       &+ \sum_{F\in\facets}h^{-1}\big\|\Pi^R_F\jump{u_h}_t\big\|_F^2
  \sim
  \sum_{T\in\mesh} \|\eps(u_h)\|_T^2 + \sum_{F\in\facets}
      h^{-1}\big\|\Pi^0_F\jump{u_h}_t\big\|_F^2 
      + h\big\|\Pi^0_F\jump{n \cdot \curl(u_h)}\big\|_F^2
\end{align*}
introduced together with the Korn inequality in \cite[Lemma 3.1]{gopalakrishnan2021minimal}.

\section{Numerical results} \label{sec::numerical}

We now present numerical results that were achieved using the Netgen/NGSolve
meshing and Finite Element software, \cite{netgen, cpp_ngs}, and
the Algebraic Multigrid extension library NgsAMG, \cite{Kogler2021},
available from \cite{ngs_webpage, ngs_amg_github}.
The computations were performed on the Vienna Scientific Cluster (VSC4).

We considered two problems,
the first of which is a standard benchmark problem from literature where we investigate
the relative performance of different ASP variations
and demonstrate robustness in the polynomial degree.
The second problem is a flow around an airplane model and is meant to
demonstrate the effectiveness of the method even in less academic situations.

For both cases, the viscosity is fixed to $\nu = 10^{-3}$,
the preconditioner in the conforming auxiliary space $\hat{V}_h$
was given by a single Algebraic Multigrid V-cycle and we used preconditioned GMRES
with a relative tolerance of $10^{-6}$ to solve the saddle point problem.
Instead of the difficult to parallelize Block-Gauss-Seidel smoothers in $\SsMHM$ and $\SsDMHM$,
we use block versions of the scalable semi-multiplicative $\ell_1$-smoothers
from \cite{UPSM}.
We show weak scaling results and therefore aim to keep
the number of elements per core constant, however
are only able to ensure this approximately
because of the unstructured tetrahedral meshes we use,

The obtained results, listed in Tables \ref{tab::ST_smoothing} - \ref{tab::plane}
will be discussed in detail below.
For every computation we list the number of elements in the mesh $|\mesh|$
and the number of cores \#P.
With the $\Sigma_h$ dofs freedom condensed out of the system,
the relevant number of dofs is that of $\VUUH\times Q_h$
which we list as \#D.
We give the number of iterations of GMRES needed as \#IT
and the total time to solution $t_{\text{tot}}$ in seconds
as well as the separate times for setting up $t_{\text{sup}}$
and solving $t_{\text{sol}}$ the systems, all excluding the time for loading the mesh.

\subsection{Flow around a cylinder}
This first series of computations concerns the flow around a cylinder as in \cite{Schaefer1996}.
The cuboid-shaped channel $\Omega$ with cylindrical obstacle $\Omega_{c}$,
$\Omega := (0,2.5) \times (0,0.41) \times (0,0.41) \setminus \overline \Omega_{c}$
is depicted on the left in Figure \ref{fig::tunnelandplane}.
The boundary parts are $\GammaN=\emptyset$, $\GammaNT=\{(2.5,y,z)\in\partial\Omega\}$
with $\GammaD = \Gamma_{\textrm{in}}\cup\Gamma_{\textrm{wall}}$ split into
inflow boundary
$\Gamma_{\text{in}} := \{ (0,y,z) \in \partial\Omega\}$,
where we impose a parabolic velocity inflow
and wall boundary $\Gamma_{\textrm{wall}}$ with homogenous Dirichlet conditions.

\subsubsection{Full versus condense system}
We first discuss whether preconditioning $\SsM$ via $\SsDM$ as
described in Section \ref{ssec::schur_schur} is purely convenient for
theory or also advantageous in practice. For that, we compare the
multiplicative ASPs over a range of problem sizes and fixed polynomial
degree $\pdg=2$. As can be clearly seen in Table
\ref{tab::ST_full_vs_cond}, preconditioning via the condense system
leads to considerably better performance and is the approach we take
from here on out.

\subsubsection{Additive versus multiplicative ASP}
The second choice is between additive and multiplicative ASPs,
we again fix the polynomial degree to $\pdg=2$ for the comparison
in Table \ref{tab::ST_smoothing}.
From the results it is once again clear that
the multiplicative preconditioner is superior and our method of choice going forward.



\subsubsection{High order robustness}
Finally, we demonstrate robustness in the polynomial degree $\pdg$ with
results for $\pdg\in\{1,2,4\}$.
Our choice of preconditioner, informed by previous results, is the multiplicative
ASP for the condense system, this time with two smoothing steps.
Due to considerably increased memory requirements, different meshes were
used for $\pdg=4$ than for $\pdg=1,2$.

\begin{table}
  \scriptsize
  \centering 
  \begin{tabular}{
    r 
    c 
    c|| 
    c| 
    c 
    c 
    c|| 
    c| 
    c 
    c 
    c 
    }
    \multicolumn{3}{c||}{} &
    \multicolumn{4}{c||}{\text{Full system}} & 
    \multicolumn{4}{c}{\text{Condense system}} \\ 
    \midrule
    \midrule
    $|\mesh|$
    &\#D
    &\#P
    &\#IT
    &$t_{\text{tot}}$
    &$t_{\text{sup}}$
    &$t_{\text{sol}}$
    &\#IT
    &$t_{\text{tot}}$
    &$t_{\text{sup}}$
    &$t_{\text{sol}}$\\
    \midrule
    \num{55248}&\num{1918320}&7&166&67.2&9.0&58.2&76&29.3&9.5&19.8\\
    \num{181351}&\num{6267214}&19&119&58.2&10.2&48.0&63&32.4&11.3&21.1\\
    \num{310272}&\num{10644864}&36&235&163.1&9.1&154.0&92&59.1&10.2&48.9\\
    \num{792940}&\num{27242308}&81&128&110.6&11.8&98.8&65&53.0&12.3&40.7\\
    \num{1450808}&\num{49732592}&166&159&134.9&10.6&124.3&73&54.5&10.8&43.7\\
    \num{3535872}&\num{120857856}&408&172&171.1&11.3&159.8&78&65.3&11.8&53.5\\
    \num{6343520}&\num{216809072}&720&164&141.7&11.1&130.6&74&55.2&11.9&43.3\\
    \num{11606464}&\num{396240256}&1333&164&169.2&12.0&157.2&75&68.0&12.9&55.1\\
    \num{23176704}&\num{791367168}&2667&193&430.2&24.0&406.2&81&78.6&14.8&63.8\\
  \end{tabular}
  \caption{Comparison of multiplicative ASPs for the full system $\SsM$ and the condense system $\SsDM$
    for the channel problem with $\pdg=2$.}
  \label{tab::ST_full_vs_cond}
\end{table}

\begin{table}
  \footnotesize
  \centering 
  \begin{tabular}{
    r 
    c 
    c|| 
    c| 
    c 
    c 
    c|| 
    c| 
    c 
    c 
    c 
    }
    \multicolumn{3}{c||}{} &
    \multicolumn{4}{c||}{\text{Additive}} & 
    \multicolumn{4}{c}{\text{Multiplicative}} \\ 
    \midrule
    \midrule
    $|\mesh|$
    &\#D
    &\#P
    &\#IT
    &$t_{\text{tot}}$
    &$t_{\text{sup}}$
    &$t_{\text{sol}}$
    &\#IT
    &$t_{\text{tot}}$
    &$t_{\text{sup}}$
    &$t_{\text{sol}}$\\
    \midrule
    \num{55248}&\num{1918320}&5&191&60.9&13.9&47.0&75&46.0&14.2&31.8\\
    \num{211920}&\num{7314384}&17&191&75.8&12.8&63.0&73&45.6&14.3&31.3\\
    \num{441984}&\num{15187008}&35&206&154.4&13.4&141.0&77&65.7&14.6&51.1\\
    \num{1450808}&\num{49732592}&111&169&132.8&14.0&118.7&73&73.1&15.2&57.9\\
    \num{6343520}&\num{216809072}&480&179&176.0&15.7&160.3&74&90.9&17.0&73.9\\
    \num{13562880}&\num{462883584}&1040&209&261.5&16.7&244.8&78&120.0&18.0&101.9\\
    \num{35229696}&\num{1201324032}&2698&230&400.2&18.0&382.3&88&159.4&20.3&139.13\\
    \num{50748160}&\num{1729955008}&3876&202&301.4&18.1&283.3&84&151.9&19.4&132.5\\
  \end{tabular}
  \caption{Comparison of additive and multiplicative ASPs for $\SsDM$
    for the channel problem with $\pdg=2$.}
  \label{tab::ST_smoothing}
\end{table}

\begin{table}
  \footnotesize
  \centering 
  \begin{tabular}{
    r 
    c|| 
    c| 
    c| 
    c 
    c 
    c|| 
    c| 
    c| 
    c 
    c 
    c 
    }
    \multicolumn{2}{c||}{} &
    \multicolumn{5}{c||}{$k=1$} & 
    \multicolumn{5}{c}{$k=2$}\\ 
    \midrule
    \midrule
    $|\mesh|$
    &\#P
    &\#D
    &\#IT
    &$t_{\text{tot}}$
    &$t_{\text{sup}}$
    &$t_{\text{sol}}$
    &\#D
    &\#IT
    &$t_{\text{tot}}$
    &$t_{\text{sup}}$
    &$t_{\text{sol}}$\\
    \midrule
\num{8601}&1&\num{118677}&86&11.6&4.2&7.4&\num{306162}&49&23.9&7.9&16.0\\
\num{317028}&36&\num{4209312}&82&21.6&5.8&15.8&\num{10954848}&53&42.8&10.2&32.6\\
\num{761759}&85&\num{10035161}&79&23.7&6.6&17.2&\num{26164394}&52&54.2&12.0&42.2\\
\num{2019989}&225&\num{26575415}&81&27.4&6.9&20.5&\num{69310742}&57&61.2&12.1&49.1\\
\num{6406377}&712&\num{83843277}&82&30.4&7.3&23.2&\num{218937570}&58&67.9&12.8&55.2\\
\num{21122473}&2347&\num{276148897}&85&32.6&8.2&24.3&\num{721277578}&62&74.8&14.7&60.1\\
\num{46480267}&5165&\num{606224245}&97&44.8&12.0&32.8&\num{1584290626}&75&92.1&15.0&77.1\\
\num{64511647}&7168&\num{841627447}&98&43.1&10.2&32.9&\num{2199348070}&75&94.6&16.1&78.5\\
\num{96966427}&10775&\num{1264069699}&106&66.0&30.1&35.9\\
\midrule\midrule
    \multicolumn{2}{c||}{} &
    \multicolumn{5}{c||}{$k=4$} & 
    \multicolumn{4}{c}{} \\ 
    \midrule
\num{8601}&11&\num{1201175}&63&40.8&19.4&21.4\\
\num{45267}&57&\num{6264135}&65&48.8&22.9&25.9\\
\num{181197}&227&\num{24828710}&65&63.7&28.2&35.5\\
\num{317028}&397&\num{43311810}&67&61.4&27.0&34.4\\
\num{761759}&953&\num{103609180}&65&65.7&27.4&38.3\\
\num{1650451}&2064&\num{224483535}&66&61.7&24.8&36.9\\
\num{6406377}&8008&\num{868129755}&64&68.7&27.7&41.0\\
  \end{tabular}
  \centering
  \caption{Multiplicative ASP for $\SsDM$ with two smoothing steps for the channel problem and
    varying polynomial order $\pdg$.}
  \label{tab::ST_ws}
\end{table}

\begin{table}
  \footnotesize
  \begin{tabular}{
    r 
    c 
    c||
    c|
    c 
    c 
    c 
    }
    $|\mesh|$
    &\#D
    &\#P
    &\#IT
    &$t_{\text{tot}}$ 
    &$t_{\text{sup}}$ 
    &$t_{\text{sol}}$ \\
    \midrule\midrule
\num{918984}&\num{31754532}&63&61&68.4&24.2&44.3\\
\num{1441885}&\num{49739362}&99&41&55.8&25.1&30.7\\
\num{1939677}&\num{66786714}&132&44&67.4&26.5&40.8\\
\num{3741663}&\num{128314458}&253&41&65.7&27.4&38.3\\
\num{11535080}&\num{395053808}&876&50&72.3&25.1&47.2\\
\num{15517416}&\num{530942928}&1176&52&79.1&25.6&53.5\\
\num{29933304}&\num{1022124000}&2261&53&82.6&26.6&56.0\\
  \end{tabular}
  \caption{Results for the model airplane problem, $\pdg=2$,
    multiplicative ASP for $\SsDM$, two smoothing steps.}
  \label{tab::plane}
\end{table}

\begin{figure}[h]
  \begin{center}
    \begin{tikzpicture}
        \node[] (CG) at (0,0) {\includegraphics[width=0.5\textwidth, clip = true, trim = 0cm 0 0 0 ]{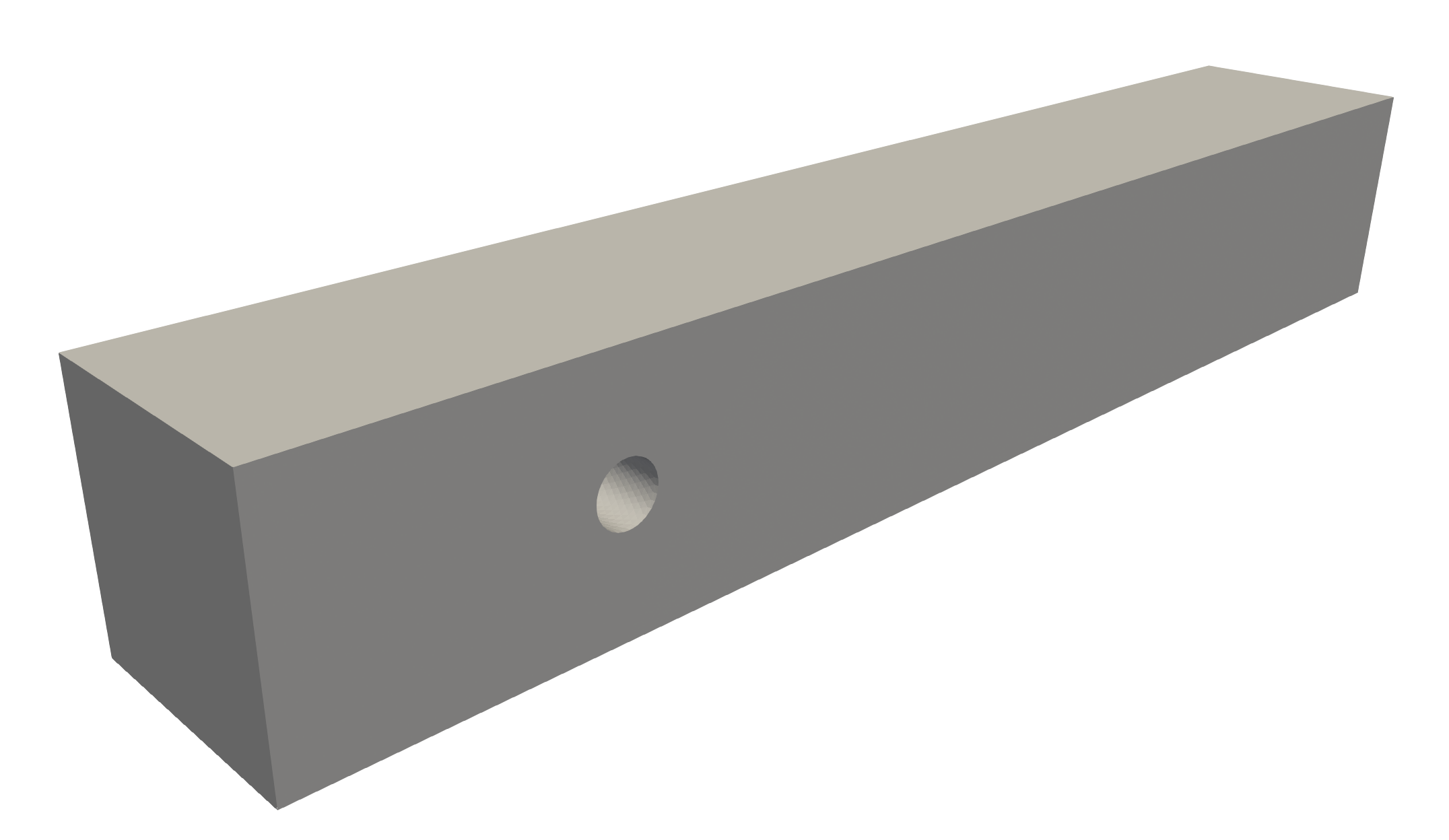}};
        \node[] (DG) at (8,0) {\includegraphics[width=0.5\textwidth, clip = true, trim = 0cm 0 0 0 ]{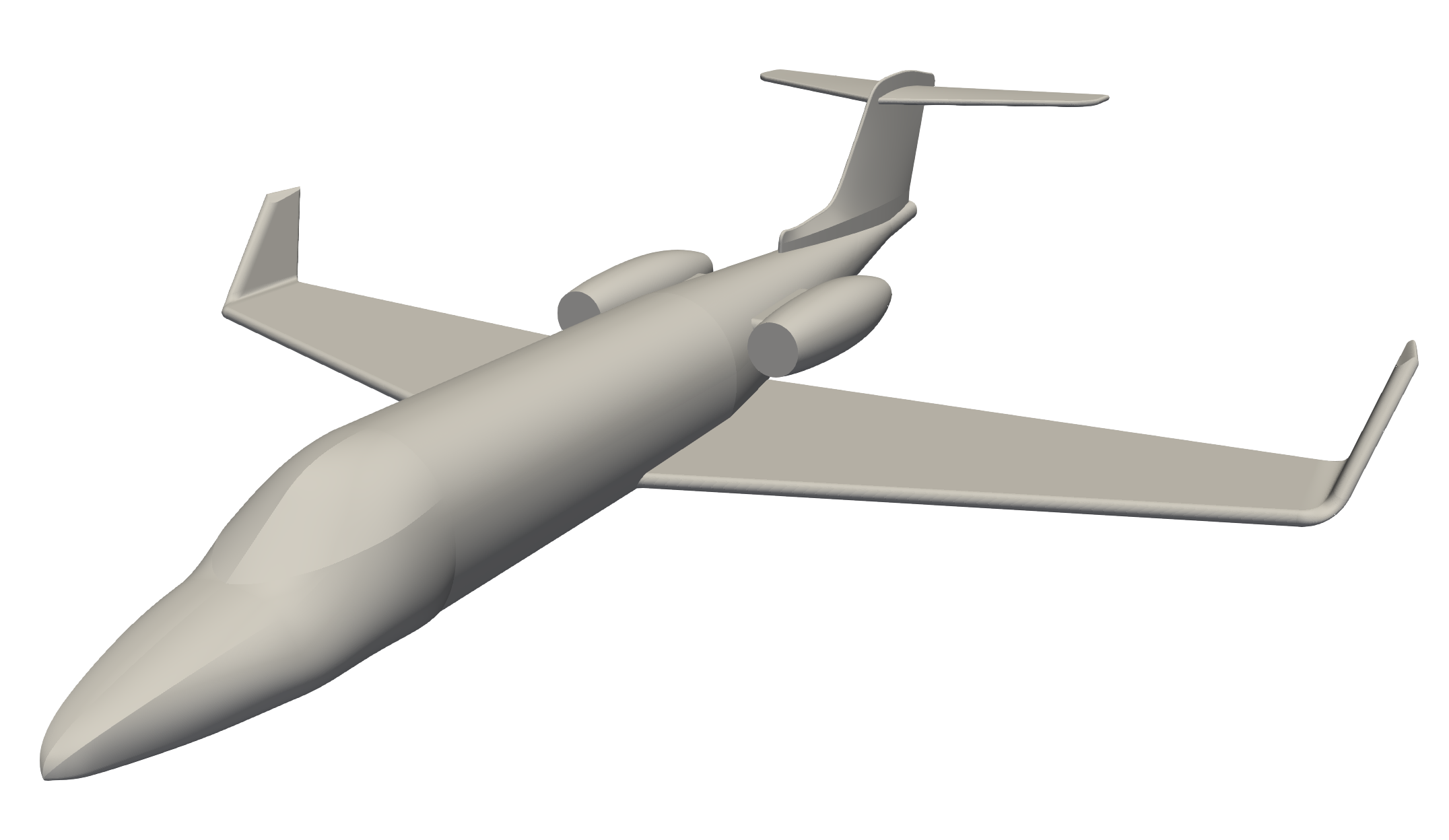}};
    \end{tikzpicture}
    \caption{Channel with cylindrical obstacle (left) and airplane model (right)}
      \label{fig::tunnelandplane}
\end{center}
\end{figure}

\subsection{Flow around an airplane model}
The computational domain $\Omega$ here
is the ``air'' in a cuboid-shaped box surrounding an airplane model $\Omega_p$ depicted in
Figure~\ref{fig::tunnelandplane},
we have $\Omega = (-8,10) \times (-7,7) \times (-3,4) \setminus \Omega_p$.
The airplane itself is contained in the bounding box $[-5.2, 5.3]\times[-4.9,4,9]\times[-0.5,1.6]$.
Boundary conditions, similar to the last case,
are imposed velocity inflow on the side of the box in front of the plane $\Gamma_{\textrm{in}}$
and homogenous Dirichlet conditions on $\partial\Omega_p$
with the rest of the boundary taken up by $\GammaNT=\partial\Omega\setminus(\Gamma_{\textrm{in}}\cup\partial\Omega_p)$.
The results can be found in Table \ref{tab::plane}.

\section{Conclusions}
\label{sec:conclusions}
In this work we introduced and analyzed a series of
auxiliary space preconditioners
for certain mass conserving mixed stress
discretizations of Stokes equations.
In the norm induced by these MCS methods,
the analysis is mostly explicit in the polynomial degree and even yields
completely explicit results in the norm induced by certain
hybrid discontinuous Galerkin
methods that feature optimal stabilization.
Numerical experiments demonstrate the robustness of the preconditioners in the polynomial degree.


\section{Acknowledgments}
The authors have been partially funded by the Austrian Science Fund
(FWF) through the research program ``Taming complexity in partial
differential systems'' (F65) - project ``Automated discretization in
multiphysics'' (P10).

\bibliography{literature}
\bibliographystyle{plain} 



\appendix

\section{Interpolation} \label{app::interp}
A standard result for $\IVf$, which, for example, follows from the Bramble Hilbert Lemma,
the discussion of jump terms arising from nodal averaging for $\Poly^1(\mesh, \rr^d)$ functions
in \cite[Section 3]{brenner_korn}, and a trace inequality is
\begin{align} \label{eq::nc_to_c_est_lin}
  \begin{aligned}
    \sum_{T\in\mesh} h^{-2} \|u - &\IVf u \|_T^2  + \|\nabla (u - \IVf u)\|_T^2 \\
    & \lesssim \sum_{T\in\mesh}\|\nabla u\|_T^2 + \sum_{F\in\facetsI}h^{-1}\|\Pi^0\jump{u}\|_F^2
  \end{aligned}
      \quad\forall u\in H^2(\mesh, \rr^d).
\end{align}
Following \cite[Section 3]{brenner_korn}, for $T\in\mesh$ we define $E_T : H^1(T, \rr^d)\rightarrow\RM(T)$ by
\begin{align*}
  \int_T (u - E_Tu)\cdot q\dx &= 0 \quad\forall q\in \Poly^0(T, \rr^d)\quad \textrm{and} \quad
  \int_T (\curl{u} - \curl{(E_Tu)})\cdot q\dx  = 0\quad\forall q\in \Poly^0(T, \rr^{d(d-1)/2}),
\end{align*}
that is $\curl{(E_T u)} = \Pi^0_T\curl{u}$ and $\Pi^0_T E_T u = \Pi^0_T u$, such that (also \cite[Section 3]{brenner_korn})
\begin{align}\label{eq::pwrb_est}
  h^{-2}\norm{u - E_T u}_{T}^2 +\norm{\nabla(u - E_T u)}_{T}^2 & \lesssim \norm{\eps(u)}_T^2.
\end{align}
With the element center of mass $x_T := \Pi^0_Tx$, elementary calculations show
\begin{align*}
  E_T u (x) = \Pi^0_T u + \kappa\big(\Pi^0_T\curl{u}\big) \cdot (x - x_T).
\end{align*}

\begin{proof}[Proof of Lemma \ref{lemma::nc_to_c_nokorn}]
  For any $T\in\mesh$, define the set of element patch elements
  $\meshT:=\{S\in\mesh:\bar{T}\cap\bar{S}\neq\emptyset\}$
  and the element patch
  $\omega_T := (\bigcup_{S\in\mathcal{T}_{h,T}}\overline{S})^{\circ}$.
  We write $\mathcal{I}_{\omega}$ for the local interpolation operator
  defined only on $\meshT$ as $\IVf$ is on $\mesh$,
  i.e. by averaging only over values from elements in $\meshT$.
  There holds $(\IVf u)_{|T} = (\mathcal{I}_{\omega}u)_{|T}$ and
  \begin{align*}
    ((I-\IVf)u)_{|T} = ( (I-\mathcal{I}_{\omega})(u - R) )_{|T} \quad\forall R\in \Poly^1(\omega_T, \rr^d)\supseteq\RM(\omega_T).
  \end{align*}
  In combination with estimate \eqref{eq::nc_to_c_est_lin} applied to $\mathcal{I}_{\omega}$ on $\meshT$ this shows
  \begin{align*}
    h^{-2}\|(I-\IVf)u\|_T^2 + \|\nabla( (I-\IVf)u )\|_T^2 
    & \leq \inf_{R\in \RM(\omega_T)} h^{-2}\|(I-\mathcal{I}_{\omega}) (u-R)\|_{\omega_T}^2 + \|\nabla((I-\mathcal{I}_{\omega}) (u-R))\|_{\omega_T}^2 \\
    & \lesssim \inf_{R\in\RM(\omega_T)} \sum_{\tilde{T}\in\meshT}\|\nabla(u-R)\|_{\tilde{T}}^2 +
      \sum_{F\in\mathcal{F}_{\omega}^\circ} h^{-1} \|\Pi^0\jump{u-R}\|_F^2 \\
    & = \inf_{R\in\RM(\omega_T)} \sum_{\tilde{T}\in\meshT}\|\nabla(u-R)\|_{\tilde{T}}^2 +
      \sum_{F\in\mathcal{F}_{\omega}^\circ} h^{-1} \|\Pi^0\jump{u}\|_F^2,
  \end{align*}
  where $\mathcal{F}_\omega^\circ$ denotes the set of interior facets of $\meshT$.
  We can further bound the volume terms by inserting $\pm E_{\tilde{T}}u$ and using \eqref{eq::pwrb_est},
  \begin{align*}
    \sum_{\tilde{T}\in\meshT}\|\nabla(u-R)\|_{\tilde{T}}^2
    & \lesssim \sum_{\tilde{T}\in\meshT}\|\nabla(u-E_{\tilde{T}} u)\|_{\tilde{T}}^2 + \|\nabla(E_{\tilde{T}} u - R)\|_{\tilde{T}}^2 
     \lesssim \sum_{\tilde{T}\in\meshT}\|\eps(u)\|_{\tilde{T}}^2 + \|\nabla(E_{\tilde{T}} u - R)\|_{\tilde{T}}^2.
  \end{align*}
  We see that it remains to find $R\in\RM(\omega_T)$ such that
  \begin{align}\label{eq::find_an_r}
    \sum_{\tilde{T}\in\meshT}\|\nabla(E_{\tilde{T}} u - R)\|_{\tilde{T}}^2 \lesssim
    \sum_{\tilde{T}\in\meshT}\|\eps(u)\|_{\tilde{T}}^2 + \sum_{F\in\mathcal{F}_\omega^\circ}h^{-1} \|\Pi^R\jump{u}\|_F^2 .
  \end{align}
  Similar to the definition of $E_T$, with $x_{\omega}:=\Pi^0_{\omega_T}x$
  a suitable $R$ is
  \begin{align*}
    R & := \Pi^0_{\omega_T} u + \kappa\big(\Pi^0_{\omega_T}\curl{u}\big)\cdot(x - x_{\omega}).
  \end{align*}
  Calculations show
  $ \Pi^0_{\omega_T} \curl{u} = \sum_{\tilde{T}\in\meshT} \alpha_{\tilde{T}} \Pi^0_{\tilde{T}}\curl{u} $
  with $\alpha_{\tilde{T}} := \frac{|\tilde{T}|}{|\omega_T|}$, and therefore
  \begin{align*}
    R & = \Pi^0_{\omega_T} u + \sum_{\tilde{T}\in\meshT}\alpha_{\tilde{T}}\kappa\big(\curl{(E_{\tilde{T}}u)}\big)\cdot(x - x_{\omega}).
  \end{align*}
  As $\eps(E_{\tilde{T}}u - R) = 0$, there holds
  $\nabla(E_{\tilde{T}} u - R) = \kappa(\curl{(E_{\tilde{T}} u - R)})\in\Poly^0(\meshT, \rr^d)$, i.e.
  \begin{align*}
    \|\nabla(E_{\tilde{T}} u - R)\|_{\tilde{T}}^2
    \sim h^d|\nabla(E_{\tilde{T}} u - R)_{|\tilde{T}}|^2
    = h^d\Big|\sum_{S\in\meshT} \alpha_S \big(\curl(E_{\tilde{T}} u) - \curl(E_{S}u)\big)\Big|,
  \end{align*}
  where we used $\sum_{S\in\meshT}\alpha_{S}=1$ and therefore, with $\alpha_S \leq 1$,
  \begin{align*}
    \sum_{\tilde{T}\in\meshT}\|\nabla(E_{\tilde{T}} u - R)\|_{\tilde{T}}^2
    & \lesssim \sum_{\tilde{T}\in\meshT} \sum_{S\in\meshT} h^d \Big|\big(\curl{(E_{\tilde{T}} u)} - \curl{(E_{S}u)}\big)\Big|.
  \end{align*}
  \begin{figure}
    \centering
    \includegraphics[width=0.3\textwidth, clip = true, trim = 0cm 0 0 0 ]{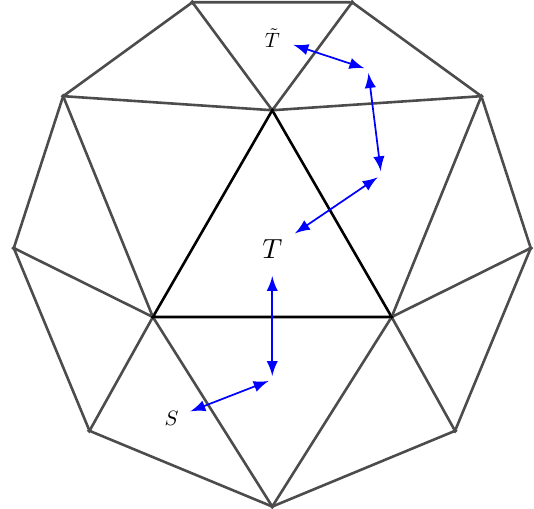}
    \caption{ Example path in two dimensions for the last estimate in \eqref{eq::patch_path}.}
    \label{fig::patch_path}
  \end{figure}
  Any two elements in $\meshT$ are connected via a path over a
  bounded number of other  elements in $\meshT$, and we can bound this last sum by one over facet terms
  (see Figure \ref{fig::patch_path}):
  \begin{align}
      \sum_{\tilde{T}\in\meshT} \sum_{S\in\meshT} h^d \Big|\big(&\curl(E_{\tilde{T}} u) - \curl(E_{S}u)\big)\Big| 
       \lesssim \sum_{F\in\mathcal{F}_{\omega}^\circ}h^d|\curl{(E_{T_{F,L}}u)} - \curl{(E_{T_{F,R}}u)}|^2,\label{eq::patch_path}
  \end{align}
  where $T_{F,L}$ and $T_{F,R}$ denote the two elements that share the facet $F$.
  Any facet is only summed up over a bounded number of times after
  reordering of the sum and due to the shape regularity of $\mesh$,
  \eqref{eq::patch_path} holds with a single constant for all patches in $\mesh$.
  These jump terms can be bounded
  an expansion of $E_{\tilde{T}}u$ at
  $x_F:= \Pi^0_F(x)$,
  \begin{align*}
    E_{\tilde{T}}u(x) = \Pi^0_F(E_{\tilde{T}}u) + \kappa(\curl{(E_{\tilde{T}}u)})\cdot (x-x_F) \quad\text{for }x\in F,
  \end{align*}
  and the elementary estimate $\|x-x_F\|_F^2\sim h^2|F|\sim h^{d+1}$.
  Writing $\mathcal{T}_F:=\{T_{F,L},T_{F,R}\}$, they show
  \begin{align*}
    h^d|\curl(E_{T_{F,L}}u) - \curl(E_{T_{F,R}}u)|^2
    & \lesssim h^{-1}\|\kappa\big(\curl(E_{T_{F,L}}u) - \curl(E_{T_{F,R}}u)\big)\cdot (x-x_F) \|_F^2 \\
    & = h^{-1}\|(\Pi^R_F-\Pi^0_F) \big(E_{T_{F,L}}u - E_{T_{F,R}}u\big)\|_F^2 \\
    & \lesssim h^{-1}\|\Pi^R_F\jump{u}\|_F^2 + \sum_{\tilde{T}\in\mathcal{T}_F} h^{-1}\|\Pi^R_F(u - E_{\tilde{T}}u)\|_{F}^2.
  \end{align*}
  An $H^1$ trace inequality and \eqref{eq::pwrb_est} let us bound
  \begin{align*}
    h^{-1}\|\Pi^R_F(u - E_{\tilde{T}}u)\|_{F}^2
    \lesssim \|(u - E_{\tilde{T}}u)\|_{H^1(\tilde{T})}^2
    \lesssim \|\eps(u)\|_{\tilde{T}}^2,
  \end{align*}
  in summary,
  \begin{align*}
    \sum_{\tilde{T}\in\meshT}\|\nabla(E_{\tilde{T}} u - R)\|_{\tilde{T}}^2
    & \lesssim \sum_{F\in\mathcal{F}_{\omega}^\circ}h^d|\curl{(E_{T_{F,L}}u)} - \curl{(E_{T_{F,R}}u)}|^2 
     \lesssim \sum_{\tilde{T}\in\meshT}\|\eps(u)\|_{\tilde{T}}^2 +  \sum_{F\in\mathcal{F}_{\omega}^\circ}h^{-1}\|\Pi^R_F\jump{u}\|_F^2,
  \end{align*}
  i.e. \eqref{eq::find_an_r} holds for our specific choice of $R$ which finishes the proof.
\end{proof}

\begin{proof}[Proof of Lemma \ref{lemma::nc_to_c_nokorn_withbc}]
  An inverse estimate for the piecewise linear $\IVf u\in\Vbar_h^f$ shows
  \begin{align*}
    \sum_{T\in\mesh} h^{-2}\|u - \IV u\|_T^2
    & \lesssim \sum_{T\in\mesh} h^{-2} \|u - \IVf u\|_T^2 + h^{-2} \|(I - \pi_0) \IVf u\|_T^2 
     \lesssim \sum_{T\in\mesh} \Big(h^{-2} \|u - \IVf u\|_T^2 + \sum_{F\in\facetsT\cap\facetsD} h^{-1} \|(I - \pi_0) \IVf u\|_F^2 \Big).
  \end{align*}
  As $(I - \pi_0) \IVf u = \IVf u$ on $F\in\facetsD$, and $\IVf u\in\Poly^1(\mesh,\rr^{\spacedim})$
  \begin{align*}
    \|(I - \pi_0) \IVf u\|_F^2
    = \|\Pi^1_F \IVf u\|_F^2
       &\lesssim \|\Pi^R_F \IVf u\|_F^2 + \|(\Pi^1_F - \Pi^R_F) \IVf u\|_F^2
      \lesssim \|\Pi^R_Fu\|_F^2 + \|\Pi^R_F(u - \IVf u)\|_F^2 + \|(\Pi^1_F - \Pi^R_F) \IVf u\|_F^2
  \end{align*}
  We bound the second one with an $H^1$ trace inequality and a scaling argument,
  \begin{align*}
    h^{-1} \|\Pi^R_F(u - \IVf u)\|_F^2 \lesssim  h^{-2} \|u - \IVf u\|_T^2 + \|\nabla (u - \IVf u)\|_T^2,
  \end{align*}
  where $T$ is the unique element such that $F\in\facetsT$.
  An explicit expansion of the piecewise linear $\IVf u$ at the facet center of mass $x_F:=\Pi^0_F x$
  shows 
  \begin{align*}
    h^{-1}\|(\Pi^1_F - \Pi^R_F) \IVf u\|_F^2 & = h^{-1}\|\eps((\IVf u)_{|T})\cdot(x - x_F)\|_F^2 \lesssim \|\eps(\IVf u)\|_T^2
     \lesssim \|\eps(u)\|_T^2 + \|\nabla(u-\IVf u)\|_T^2
  \end{align*}
  Finally, with $\jump{\Pi_F^Ru} = \Pi^R_Fu$ on $F\in\facetsD$,
  by Lemma \ref{lemma::nc_to_c_nokorn} there holds
  \begin{align*} 
    \sum_{T\in\mesh} h^{-2}\|u - \IV u\|_T^2 
    &\lesssim \sum_{T\in\mesh}  \Big( h^{-2} \|u - \IVf u\|_T^2+ \|\nabla(u - \IVf u)\|_T^2   +  \|\eps(u)\|_T^2 +
      \sum_{F\in\facetsT\cap\facetsD} h^{-1} \|\Pi_F^Ru\|_F^2 \Big)\\
    &\lesssim \sum_{T\in\mesh} \Big( \|\eps(u)\|_T^2 + \sum_{F\in\facetsT}h^{-1}\|\Pi_F^R\jump{u}\|_F^2
      \Big).
  \end{align*}
  The other volume terms $\|\nabla(u-\IV u)\|_T$ in \eqref{eq::nc_to_c_nokorn_withbc} are bounded analogously.
\end{proof}

\section{Trace estimates} \label{app::trace} 
The crucial step in the proof of \eqref{eq::scalarinverseest} in \cite{Schb_DDHDG} 
was to construct for a given $w\in\Poly^{\pdg}(T,\rr)$ a $\tilde{w}\in\Poly^{\pdg}(T,\rr)$
that approximates it in $\|\cdot\|_{j,F}$ and is bounded in the $H^1$ semi-norm and yet 
$\tilde{w}_{|\partial F}=0$.
That is, implicitly, for $F\in\facets$ and $T\in\mesh$ with $F\in\facetsT$, an operator
\begin{equation}
\begin{gathered}
  \mathcal{E}^{s,F}_0: \Poly^{\pdg}(T,\rr)\rightarrow\Poly^{\pdg}(T,\rr) 
  \quad\textrm{with}\quad(\mathcal{E}_0^{s,F}u)_{|\partial F} = 0 \\
  \textrm{and}\quad\|\nabla\mathcal{E}_0^{s,F}u\|_T^2 + \|\mathcal{E}_0^{s,F}u - u\|^2_{j,F} \lesssim \log{\pdg} \|u\|^2_{H^{1}(T)}
\end{gathered}\label{eq::Esf}
\end{equation}
was constructed.
As no boundary conditions are enforced strongly in $\|\cdot\|_{1,F,0}$, 
the trace of $\mathcal{E}^{s,F}_0w$ could be extended
to a function admissible for the infimum in $\|\cdot\|_{1,F,0}$.
In addition to $\mathcal{E}^{s,F}_0$,
of the commuting $H^1, H(\curl)$ and $H(\div)$ extensions
introduced in \cite{polext_I, polext_II, polext_III},
we need the first,
\begin{gather*}
  \mathcal{E}^{s}: \Poly^{\pdg}(\partial T,\rr)\rightarrow\Poly^{\pdg}(T,\rr) \\
  \textrm{with}\quad(\mathcal{E}^{s}u)_{\partial T} = u_{\partial T},
  \quad\textrm{and}\quad \|\mathcal{E}^{s}u\|_{H^1(T)}^2 \lesssim \|u\|_{H^{1/2}(\partial T)}^2,
\end{gather*}
see \cite[Theorem 6.1]{polext_I}, and last,
\begin{gather*}
  \mathcal{E}^{\div}: \{ u\in H^{-1/2}(\partial T) : u_{|F}\in\Poly^{\pdg}(F,\rr^d)~\forall F\in\facetsT\}\rightarrow\Poly^{\pdg}(T,\rr^d) \\
  \textrm{with}\quad((\mathcal{E}^{\div}u)_n)_{|F} = (u_n)_{|F}~F\in\facetsT,
  \quad\textrm{and}\quad \|\mathcal{E}^{\div}u\|_{H^1(T)}^2 \lesssim \|u\|_{H^{-1/2}(F)}^2.
\end{gather*}
Note although that $\mathcal{E}^{\div}$ was constructed from
$H^{-1/2}(\partial T)\rightarrow H^{\div}(T)$ in \cite[Theorem 7.1]{polext_II},
the authors actually proved continuity in the $H^1(T)$ norm.

\begin{lemma}\label{lemma::stabfacettrace}
  For $F\in\facets$ and $u\in \Poly^{\pdg}(T,\rr^{\spacedim})$ there exists $\tilde{u}\in \Poly^{\pdg}(T,\rr^{\spacedim})$
  with $\tilde{u}_n = u_n$ on $F$ and $\tilde{u}_n=0$ on $\tilde{F}\in\facetsT\setminus\{F\}$
  such that
  \begin{align}\label{eq::stabfacettrace}
    \|\nabla \tilde{u}\|_T^2 + \|(\tilde{u} - u)_t\|_{j,F}^2
    + \sum_{\tilde{F}\in\facetsT\setminus\{F\}}h^{-1}\|u_t\|_{j,\tilde{F}}^2
    \lesssim
    (\log{\pdg})^3\|u\|_{H^1(T)}^2.
  \end{align}
\end{lemma}

\begin{proof} 
  To simplify the notation we only show a proof in three dimensions, 
  the two-dimensional case works analogously.
  It also suffices to show the estimate on a reference tetrahedron,
  for general tetrahedra it follows from scaling arguments.

  First, we use $\mathcal{E}^{\div}$ to get a $\tilde{u}_1$ 
  which fulfills $(\tilde{u}_1)_n = u_n$ on $F$ and
  $(\tilde{u}_1)_n = 0$ on $\tilde{F}\in\facetsT\setminus\{F\}$
  with bounded $H^1$ norm,
  \begin{align*}
    \|\tilde{u}_1\|_{H^1(T)}^2
    \lesssim \|u\|_{H^{-1/2}(F)}^2
    \lesssim \|u\|_{H(\div,T)}^2
    \leq \|u\|_{H^1(T)}^2.
  \end{align*}
  We have no control over the tangential traces of
  $\tilde{u}_1$ and have to add correction terms. For all
  $\hat{F}\in\facetsT$ we pick two arbitrary normalized, orthogonal
  tangent vectors $t_{\hat{F}}$ and $\tilde{t}_{\hat{F}}$
  and write the ``errors'' we need to compensate for on $F$ and $\tilde{F}\in\facetsT\setminus\{F\}$ as
  \begin{alignat*}{3}
    &\lambda_F t_{F}:= \big((u-\tilde{u}_1) \cdot t_{F})t_{F},
    && \quad\quad \tilde{\lambda}_F\tilde{t}_{F} := \big((u-\tilde{u}_1)\cdot \tilde{t}_{F}\big)\tilde{t}_{F}, \\
    &\lambda_{\tilde{F}}t_{\tilde{F}} := -\big(\tilde{u}_1\cdot t_{\tilde{F}}\big)t_{\tilde{F}},
    && \quad\quad \tilde{\lambda}_{\tilde{F}}\tilde{t}_{\tilde{F}} := -\big(\tilde{u}_1\cdot \tilde{t}_{\tilde{F}}\big)\tilde{t}_{\tilde{F}}.
  \end{alignat*}
  We write $\mathcal{E}^{s,\hat{F}}$ for $\mathcal{E}^s$ applied to the
  extension by zero to $\partial T$ of functions that vanish on
  $\partial\hat{F}$. This defines an $H_{00}^{1/2}$ stable extension
  \begin{align*}
    \mathcal{E}^{s,\hat{F}} : \Poly^{\pdg}(F,\rr)\cap H_{00}^{1/2}(F) \rightarrow \Poly^{\pdg}(T,\rr)
    \quad\textrm{with}\quad \|\mathcal{E}^{s,\hat{F}}u\|_{H^1(T)} \lesssim \|u\|_{H_{00}^{1/2}(F)}
  \end{align*} 
  and construct a corrected $\tilde{u}$ as
  \begin{align*}
    \tilde{u} ~:=~& \tilde{u}_1
    + \sum_{\hat{F}\in\facetsT}
    \mathcal{E}^{s,\hat{F}}\mathcal{E}^{s,\hat{F}}_0(\lambda_F)t_{\hat{F}}
    + \mathcal{E}^{s,\hat{F}}\mathcal{E}^{s,\hat{F}}_0(\tilde{\lambda}_F)\tilde{t}_{\hat{F}},
  \end{align*}
  where we understand $\mathcal{E}^{s,\hat{F}}$ to be applied to the respective trace on $\hat{F}$.
  The added corrections are normal bubbles because on their associated facet they are a scalar times a tangential
  and their trace vanishes on all others, i.e. $\tilde{u}_n = (\tilde{u}_1)_n~\forall F\in\facetsT$
  and $\tilde{u}$ is admissible and we need to show that if fulfills \eqref{eq::stabfacettrace}.

  As $\mathcal{E}^{s,\tilde{F}}$ restricted to $\tilde{F}$ is just the identity,
  for $\tilde{F}\in\facetsT\setminus\{F\}$ there holds
  \begin{align*}
    \|\tilde{u}\cdot t_{\tilde{F}}\|_{j,\tilde{F}}^2
    = \|(\tilde{u}_1\cdot t_{\tilde{F}}) - \mathcal{E}^{s,\tilde{F}}\mathcal{E}^{s,\tilde{F}}_0(\tilde{u}_1\cdot t_{\tilde{F}})\|_{j,\tilde{F}}^2
    = \|(I - \mathcal{E}^{s,\tilde{F}}_0)(\tilde{u}_1\cdot t_{\tilde{F}})\|_{j,\tilde{F}}^2
  \end{align*}
  and \eqref{eq::Esf} implies
  \begin{align*}
    \|\tilde{u}\cdot t_{\tilde{F}}\|_{j,\tilde{F}}^2
    \lesssim \log{\pdg}\|\tilde{u}_1\cdot t_{\tilde{F}}\|_{H^1(T)}^2
    \lesssim \log{\pdg}\|\tilde{u}_1\|_{H^1(T)}^2
    \lesssim \log{\pdg}\|u\|_{H^1(T)}^2.
  \end{align*}
  The volume terms arising from the correction of $\lambda_{\tilde{F}}$
  for $\tilde{F}\in\facetsT\setminus\{F\}$ can be bounded
  with the inverse estimate $\|v\|_{H^{1/2}_{00}(\tilde{F})}\lesssim (\log{\pdg})^2 \|v\|_{H^{1/2}(\tilde{F})}$
  for polynomials that vanish on $\partial\tilde{F}$, see \cite[Lemma 4.7]{bica_p_substr},
  \begin{align*}
    \|\nabla \mathcal{E}^{s,\tilde{F}}\mathcal{E}^{s,\tilde{F}}_0(\lambda_{\tilde{F}})t_{\tilde{F}}\|_T^2
    \lesssim \|\mathcal{E}^{s,\tilde{F}}_0\lambda_{\tilde{F}}\|_{H^{1/2}_{00}(\tilde{F})}^2
    \lesssim (\log{\pdg})^2\|\mathcal{E}^{s,\tilde{F}}_0\lambda_{\tilde{F}}\|_{H^{1/2}(\tilde{F})}^2,
  \end{align*}
  and we can continue with \eqref{eq::Esf} to see
  \begin{align*}
    \|\nabla \mathcal{E}^{s,\tilde{F}}\mathcal{E}^{s,\tilde{F}}_0(\lambda_{\tilde{F}})t_{\tilde{F}}\|_T^2
    \lesssim (\log{\pdg})^3\|\tilde{u}_1\cdot t_{\tilde{F}}\|_{H^{1}(T)}^2
    \lesssim (\log{\pdg})^3\|u\|_{H^{1}(T)}^2.
  \end{align*}
  Analogously, we show these same bounds for volume and trace terms for $\tilde{F}=F$ as well as
  $\tilde{t}_{(\cdot)}, \tilde{\lambda}_{(\cdot)}$ instead of $t_{(\cdot)}, \lambda_{(\cdot)}$.
  In summary, we have
  \begin{align*}
    \|\nabla\tilde{u}\|_T^2 &+ \|(u - \tilde{u})_t||_{j,F}^2 + \sum_{\tilde{F}\in\facetsT\setminus\{F\}} \|\tilde{u}_t\|_{j,\tilde{F}}^2 \\
    \lesssim~& \|\nabla\tilde{u}_1\|_F^2
               + \sum_{\hat{F}\in\facetsT}
               \|\nabla \mathcal{E}^{s,\hat{F}}\mathcal{E}^{s,\hat{F}}_0(\lambda_{\hat{F}})t_{\hat{F}}\|_T^2
               + \|\nabla \mathcal{E}^{s,\hat{F}}\mathcal{E}^{s,\hat{F}}_0(\tilde{\lambda}_{\hat{F}})\tilde{t}_{\hat{F}}\|_T^2 
              + \|(u - \tilde{u})_t||_{j,F}^2 + \sum_{\tilde{F}\in\facetsT\setminus\{F\}} \|\tilde{u}_t\|_{j,\tilde{F}}^2 \\
    \lesssim~& (\log{\pdg})^3 \|u\|_{H^{1}(T)}^2.
  \end{align*}
\end{proof}

\begin{proof} [Proof of Lemma~\ref{coro::traceZ}]
  For the minimizer $w$ in \eqref{eq::traceF},
  a Korn inequality on $T$ shows
  \begin{align*}
    \|w\|_{H^1(T)}^2
    & \lesssim \|\eps(w)\|_T^2 + \|\Pi^R_F w\|_{j,F}^2 
     = \|\eps(w)\|_T^2 + \|\Pi^R_F (w - \hat{u})_t\|_{j,F}^2 
     \lesssim \|(u, \hat{u})\|_{\eps,F}.
  \end{align*}
  Choosing $\tilde{w}\in\Poly^{\pdg}(T,\rr^d)$ as in Lemma \ref{lemma::stabfacettrace},
  finishes the proof as it is admissible for the infimum in \eqref{eq::traceFZ}
  and bounds it by $\|w\|_{H^1(T)}^2\lesssim \|(u, \hat{u})\|_{\eps,F}$.
\end{proof}


\end{document}